\documentclass[fontsize=11pt,DIV=12]{scrartcl}

\usepackage{lmodern} % Schrift: Latin Modern

\usepackage{epsf}
\usepackage[latin1]{inputenc}
\usepackage[T1]{fontenc}
\usepackage[intlimits]{amsmath}
\usepackage{amssymb}
\usepackage{amsthm}
\usepackage{amsfonts}
\usepackage{amscd}
\usepackage{yfonts}
\usepackage{bbm}
\usepackage{url}

\allowdisplaybreaks[2] %Formeln d�rfen umgebrochen werden

\newtheorem{theorem}{Theorem}[section]
\newtheorem{lemma}[theorem]{Lemma}
\newtheorem{proposition}[theorem]{Proposition}

\theoremstyle{definition}
\newtheorem{remark}{Remark}

\newcommand{\Ek}[2][]{\ensuremath{{}^{\kappa}\mathbb{E}_{#1}\left( {#2} \right)}}
\newcommand{\Pk}[2][]{\ensuremath{{}^{\kappa}\mathbb{P}_{#1} \left( {#2} \right)}}
\newcommand{\Pkappa}[1][]{\ensuremath{{}^\kappa\mathbb{P}_{#1}}}
\newcommand{\Ekappa}[1][]{\ensuremath{{}^\kappa}\mathbb{E}_{#1}}
\newcommand{\Pkfs}[1][]{\ensuremath{{}^\kappa}\mathbb{P}_{#1}\text{-a.s.}}

\DeclareMathOperator{\Probk}{{}^\kappa\mathbb{P}}

\DeclareMathOperator{\E}{\mathbb{E}}
\DeclareMathOperator{\Prob}{\mathbb{P}}

\renewcommand{\P}[2][]{\ensuremath{\mathbb{P}_{#1} \left( {#2} \right)}}
\newcommand{\Pkern}[1][]{\ensuremath{{}^{ #1}\!P}}
\newcommand{\Phatkern}[1][]{\ensuremath{{}^{ #1}\!\widehat{P}}}
\newcommand{\Pfs}[1][]{\ensuremath{\mathbb{P}_{#1}\text{-a.s.}}}
\newcommand{\Erw}[2][]{\ensuremath{\mathbb{E}_{#1} \left( {#2} \right)}}

\DeclareMathOperator{\N}{\mathbb{N}}
\DeclareMathOperator{\Z}{\mathbb{Z}}

\DeclareMathOperator{\R}{\mathbb{R}}

\renewcommand{\S}{\mathcal{S}}

\DeclareMathOperator{\llam}{\lambda\hspace{-5.1pt}\lambda}

\DeclareMathOperator{\supp}{\mathrm{supp}}

\newcommand{\m}{\mathfrak{m}}

\renewcommand{\epsilon}{\varepsilon}
\renewcommand{\rho}{\varrho}

\newcommand{\esl}[1]{\ensuremath{\left( #1 \right)^\sim}}

\newcommand{\1}[1][]{\mathbf{1}_{#1}}
\newcommand{\norm}[1]{\ensuremath{\left\| {#1} \right\|}}
\newcommand{\abs}[1]{\ensuremath{\left| {#1} \right|}}

\renewcommand{\k}[1][0]{\ensuremath{ {\kappa_{#1}}}}

\newcommand{\Nhat}{\widehat{N}}
\allowdisplaybreaks[2] %Formeln d�rfen umgebrochen werden

\begin{document}

\title{Tail behavior of stationary solutions of random difference equations: the case of regular matrices}

\author{Gerold Alsmeyer$^{\dagger}$\thanks{Corresponding author. Email: gerolda@math.uni-muenster.de} \ and Sebastian Mentemeier\thanks{{Research supported by the Deutsche Forschungsgemeinschaft (SFB 878)}}\\\vspace{6pt}
{\em{Institut f\"{u}r Mathematische Statistik, Einsteinstr. 62, 48149 M\"{u}nster, DE}}\\
 }

\maketitle

\begin{abstract}
Given a sequence $(M_{n},Q_{n})_{n\ge 1}$ of i.i.d.\ random variables with generic copy $(M,Q)$ such that $M$ is a regular $d\times d$ matrix and $Q$ takes values in $\R^{d}$, we consider the random difference equation (RDE)
$R_{n}=M_{n}R_{n-1}+Q_{n}$, $n\ge 1$.
Under suitable assumptions stated below, this equation has a unique stationary solution
$R$ such that, for some $\kappa>0$ and some finite positive and continuous function $K$ on $S^{d-1}:=\{x\in\R^{d}:|x|=1\}$,
$$ \lim_{t\to\infty}\,t^{\kappa}\Prob(xR>t)=K(x)\quad\text{for all }x\in S^{d-1} $$
holds true. A rather long proof of this result, originally stated by Kesten at the end of his famous article \cite{Kesten1973}, was given by LePage \cite{LePage1983}. The purpose of this article is to show how regeneration methods can be used to provide a much shorter argument (in particular for the positivity of K). It is based on a multidimensional extension of Goldie's implicit renewal theory developed in \cite{Goldie1991}.

\vspace{.5cm}\noindent
\emph{Keywords}: Markov renewal theory; implicit renewal theory, Harris recurrence, regeneration, random operators and equations; stochastic difference equations; random dynamical systems 

\vspace{.1cm}\noindent
\emph{AMS 2010 Subject Classification} 60K05, 60H25, 39A50, 37H10
%\begin{classcode}	%AMS 2010
%60K05, % renewal theory
%60H25 % 	Random operators and equations
%39A50, % Stochastic difference equations
%37H10 % Random dynamical systems: Generation, random and stochastic difference and differential equations
%
%\end{classcode}\bigskip
\end{abstract}

\section{Introduction} \label{sec:Intro}

Let $(M_{n},Q_{n})_{n\ge 1}$ be a sequence of i.i.d.\ random variables with generic copy $(M,Q)$ such that $M$ is a real $d\times d$ matrix and $Q$ takes values in $\R^{d}$. Suppose further that
\begin{equation}\label{logmom M}\tag{A1}
\E\log^{+}\|M\|<\infty
\end{equation}
where $\|M\|:=\sup_{|x|=1}|xM|$. Then, with $\Pi_{n}:=M_{1}\cdot...\cdot M_{n}$, there exists $\beta \in [-\infty, \infty)$ such that
\begin{equation*} %\label{def Liapunov exponent}
\beta:=\lim_{n\to\infty}n^{-1}\log\|\Pi_{n}\|\quad\Pfs
\end{equation*}
and defines the Liapunov exponent of the RDE
\begin{equation}\label{RDE}
R_{n}=M_{n}R_{n-1}+Q_{n},\quad n\ge 1.
\end{equation}
If $\beta$ is negative and
\begin{equation}\label{logmom Q}\tag{A2}
\E\log^{+}|Q|<\infty,
\end{equation}
then this recursive Markov chain has a unique stationary distribution which is given by the law of the almost surely convergent series
\begin{equation}\label{solution to RDE}
R:=\sum_{n\ge 1}\Pi_{n-1}Q_{n}
\end{equation}
and is also characterized as the unique solution to the stochastic fixed-point equation (SFPE)
\begin{equation}
Y\stackrel{d}{=}MY+Q
\end{equation}
where $\stackrel{d}{=}$ means equality in law and where $Y$ is understood to be independent of $(M,Q)$. This by now standard result may easily be deduced from a more general one for iterations of random Lipschitz maps, see e.g.\ \cite{Elton1990} or \cite{DiacFreed1999}. Our concern here is the tail behavior of $R$ in the case when
$M$ takes almost surely values in $GL(d,\R)$, the group of regular $d\times d$ matrices with real entries.

%We call a matrix $A$ \emph{feasible} if it belongs to the support of $\Prob(\Pi_n\in\cdot)$ for some $n$ and has an (algebraically) simple eigenvalue $\lambda_1 >0$ which exceeds all other eigenvalues of $A$ in absolute value. 
For $x \in \R^d\backslash\{0\}$, we write $x^{\sim}$ for its projection on the unit sphere $S:=S^{d-1}$, thus $x^{\sim} := |x|^{-1} x$. Lebesgue measure on the space of real $d \times d$-matrices, seen as $\R^{d^2}$, is denoted as $\llam$ and the uniform distribution on $S$ as $\llam_{S}$. Finally, the open $\delta$-balls in $S$ and $GL(d,\R)$ with centers $x$ and $A$ are denoted as $B_{\delta}(x)$ and $B_{\delta}(A)$, respectively.

\begin{theorem} \label{theorem:main}
Consider the RDE \eqref{RDE} and suppose that, in addition to \eqref{logmom M}, \eqref{logmom Q} and $\beta<0$, the following assumptions hold:
\begin{align}
&\Prob(M\in GL(d,\R))=1.\label{A3}\tag{A3}\\
&\max_{n\ge 1}\,\Prob((x\Pi_{n})^{\sim}\in U)>0\text{ for any $x\in S$ and any open $\emptyset\ne U\subset S$}.\label{A4}\tag{A4}\\
&\Prob(\Pi_{n_{0}}\in\cdot)\ge\gamma_{0}\1[{B_{c}(\Gamma_{0})}]\llam\text{ for some $\Gamma_{0}\in GL(d,\R)$, $n_{0}\in\N$ and $c,\gamma_{0}>0$}.\label{A5}\tag{A5}\\
%&\text{There exists a feasible matrix}.\label{A6}\tag{A6}\\
&\Prob(Mv+Q=v)<1\text{ for any column vector }v\in\R^{d}.\label{A6}\tag{A6}\\
&\text{There exists $\kappa_{0}>0$ such that}\notag\\
&\quad\E\inf_{x\in S}|xM|^{\kappa_{0}} \geq 1,\ \E\|M\|^{\kappa_{0}}\log^{+}\|M\|<\infty\text{ and }0<\E|Q|^{\kappa_{0}}<\infty.\label{A7}\tag{A7}
%&\quad\E\lambda_d (M^{\top}M)^{\kappa_{0}/2} \geq 1,\ \E\|M\|^{\kappa_{0}}\log^{+}\|M\|<\infty\text{ and }0<\E\|Q\|^{\kappa_{0}}<\infty.\label{A7}\tag{A7}
\end{align}
Then there exists a unique  $\kappa \in (0, \kappa_0]$ such that
\begin{equation}\label{theorem:main:assertion1}
	\lim_{n \to \infty} n^{-1} \log \E{ \norm{\Pi_n}^\kappa}=0,
\end{equation}
and
\begin{equation}\label{theorem:main:assertion2}
\lim_{t \to \infty} t^{\kappa}\,\P{xR > t} = K(x)\quad\text{for all }x\in S,
\end{equation}
where $K$ is a finite positive and continuous function on $S$.
\end{theorem}

\begin{remark}
This result (with one extra condition) was stated by Kesten at the end of his famous article \cite{Kesten1973} and later proved by Le Page \cite{LePage1983} with the help of Kesten's Markov renewal theorem \cite{Kesten1974} (and without assuming \eqref{A5}). Markov renewal theory also plays an essential role in our approach, but we make use of a different Markov renewal theorem taken from \cite{Als1997} to show how proofs can be shortened considerably using Harris recurrence, which is the primary intention of this article. Condition \eqref{A5} plays a crucial role in obtaining \eqref{theorem:main:assertion2} for all $x\in S$. Not assumed by LePage, he instead imposes the extra condition
$$ \E|Q|^{\kappa_{0}+\epsilon}<\infty\quad\text{for some }\epsilon>0 $$
to arrive at the same conclusion. A similar result was also derived by Kl\"uppelberg and Pergamenchtchikov \cite{Klueppelberg2004} for a more specialized model. As further references, we mention related work by de Saporta et al.\ \cite{deSaporta2004}, by Guivarc'h \cite{Gui2006}  and, most recently, by Buraczewski et al.\ \cite{Buracz2009} who obtain more precise information on the tails of $R$ under the restriction that $M$ is a similarity (product of a dilation and an orthogonal transformation).
\end{remark}

\begin{remark}
The most interesting ingredient to our approach may be roughly described as a suitable combination of Goldie's implicit renewal theory \cite{Goldie1991}, lifted to the multidimensional situation, with the technique of sampling along ''nice'' regeneration epochs for the considered RDE (see Sections \ref{sect:implicitMRT}, \ref{sect:remove null set} and \ref{sect:positive}).
\end{remark}

%\begin{remark}
%The Markov renewal theorem actually yields assertion \eqref{theorem:main:assertion2} for $\pi$-almost all $x$, where $\pi$ denotes the stationary distribution of an intrinsic Markov chain $(X_n)_{n\ge 0}$ on the sphere (see Section \ref{sect:tailMat}). But it is shown in Lemma \ref{pi:equivalent:Lebesgue} at the end of this work that $\pi$ is equivalent to $\llam_{S}$.
%\end{remark}

\begin{remark}
Assumption \eqref{A6} is a condition on the dependence of $M$ and $Q$ (there is no need for independence), and asserts particularly that no Dirac measure solves the RDE. In fact our assumptions assure a priori, that $\supp R$ is unbounded in $\R^d$ (see Lemma \ref{lem:R:unbounded}).
\end{remark}

\begin{remark}
Note that condition $\E\inf_{x\in S}|xM|^{\kappa_{0}} \geq 1$ in \eqref{A7} may be restated as
$$ \E\lambda_d (MM^{\top})^{\kappa_{0}/2} \geq 1, $$
where $\lambda_d(MM^{\top})$ denotes the smallest of the $d$ eigenvalues of the symmetric matrix $MM^{\top}$. This follows because $|xM|=(xMM^{\top}x^{\top})^{1/2}$.
\end{remark}

The further organization is as follows: Section \ref{section:MC} discusses the two central assumptions \eqref{A4} and \eqref{A5} in terms of their implications for obtaining Harris recurrence of an intrinsic Markov chain $(X_n)_{n\ge 0}$ on the sphere (see Section \ref{sect:tailMat}). We then proceed in Section \ref {stopped RDE} with some useful results concerning a whole class of SFPE that are solved by $R$ and obtained via the use of stopping times. In particular, we explain how geometric sampling allows us to simplify some assumptions in Theorem \ref{theorem:main} before proving it. Section \ref{sect:MRT} collects some facts about Harris recurrence and Markov renewal theory which are used in Section \ref{sect:tailMat} to show that $\lim_{ t \to \infty} t^{-\kappa}\,\P{ \sup_{n \in \N} \abs{xM_1 \cdots M_n} > t}$ exists and is positive. This section further contains all necessary ingredients for the Markov renewal approach including the crucial measure change (harmonic transform) also used by Kesten. The proof of Theorem \ref{theorem:main} is then provided in Sections \ref{sect:implicitMRT}, \ref{sect:remove null set} and \ref{sect:positive}.

\section{Minorization: Implications of \eqref{A4} and \eqref{A5}}\label{section:MC}

It is useful to discuss at this early point the implications of the two conditions \eqref{A4} and \eqref{A5} in terms of the semigroup $(P^{n})_{n\ge 1}$ of Markov transition kernels on $S$, defined by $P^{n}(x,A):=\Prob((x\Pi_{n})^{\sim}\in A)$ for $x\in S$ and measurable $A\subset S$. The pertinent Markov chain being of interest here will be introduced in Section \ref{sect:tailMat}. For compact subsets $C$ of $GL(d,\R)$, we further define the substochastic kernels $P_{C}^{n}(x,\cdot):=\Prob((x\Pi_{n})^{\sim}\in\cdot,\Pi_{n}\in C)$. Let $I$ denote the identity matrix.

\begin{lemma}\label{lem:MC}
Suppose \eqref{A4} and \eqref{A5}. Then there exists a compact subset $C$ of $GL(d,\R)$ such that, for each $x\in S$, there are $\delta,p>0$, $m\in\Bbb{N}$, and a probability measure $\phi$ on $B_{\delta}(x)$ satisfying
\begin{equation}\label{MC}
P^{m}(y,\cdot) \ge P_{C}^{m}(y,\cdot) \ge p\phi\tag{MC}
\end{equation} 
for all $y\in S$.
\end{lemma}

\begin{proof}
Fix any $x\in S$. By \eqref{A4} and \eqref{A5}, we can choose $n_{1}\ge 1$, $\eta>0$ and thereupon $0<\delta< \eta$, $0 < \varsigma < c$ and a compact $B_1\subset GL(d,\R)$ in such a way that
\begin{enumerate}
\item[(i)] $\zeta:=\inf_{y\in B_{\delta}(x)}P_{B_1}^{n_{1}}(y,U)>0$, where $U:=B_{\eta}\left((x\Gamma_{0}^{-1})^{\sim}\right)$;
\item[(ii)] $\Phi(A):=\int_{B_{\varsigma}(I)}\1[A]((x\mathfrak{m})^{\sim})\llam(d\mathfrak{m})$ defines nonzero measure on $B_{\delta}(x)$;
%, thus
%$$ \int_{B_{\delta}(I)}\1[B_{\delta}(x)]((x\mathfrak{m})^{\sim})\,\llam(d\mathfrak{m}) > 0. $$
\item[(iii)] for any $u\in U$, we have $u=(xF_{u}\Gamma_{0}^{-1})^{\sim}$ as well as $F_{u}\Gamma_{0}^{-1}B_{c}(\Gamma_{0})\supset B_{\varsigma}(I)$ for some unitary matrix $F_{u}$.
\end{enumerate}
Put $C:=B_1\cdot \overline{B_{c}(\Gamma_{0})}:=\{\Lambda_{1}\Lambda_{2}:\Lambda_{1}\in B_1,\Lambda_{2}\in\overline{B_{c}(\Gamma_{0})}\}$, which is a compact subset of $GL(d,\R)$ (as the continuous image of the compact $B_1\times \overline{B_{c}(\Gamma_{0})}$).
It then follows for any $y\in B_{\delta}(x)$ and measurable $A\subset S$ that
\begin{align}\label{minorization estimate}
P_{C}^{n_{0}+n_{1}}(y,A) &\ge \int_{U}\int_{B_{c}(\Gamma_{0})} \1[A]((u\mathfrak{m})^\sim)\,\Prob(\Pi_{n_{0}}\in d\mathfrak{m})\,P_{B_1}^{n_{1}}(y,du)\nonumber  \\
&\ge \gamma_{0}\int_{U}\int_{B_{c}(\Gamma_{0})} \1[A]((u\mathfrak{m})^\sim)\,\llam(d\mathfrak{m})\,P_{B_1}^{n_{1}}(y,du) \nonumber  \\
&= \gamma_{0}\int_{U}\int_{B_{c}(\Gamma_{0})} \1[A]((xF_{u}\Gamma_{0}^{-1}\mathfrak{m})^\sim)\,\llam(d\mathfrak{m})\,P_{B}^{n_{1}}(y,du) \nonumber  \\
&\ge \gamma_0|\det(\Gamma_{0})|^d\,P_{B_1}^{n_1}(y,U)\,\int_{B_{\varsigma}(I)} \1[A]((x\mathfrak{m})^\sim)\,\llam(d\mathfrak{m})\nonumber\\
&\ge \gamma_0\zeta |\det(\Gamma_{0})|^d\int_{B_{\varsigma}(I)} \1[A]((x\mathfrak{m})^\sim)\,\llam(d\mathfrak{m})
\end{align}
which proves \eqref{MC} for all $y\in B_{\delta}(x)$ with $m=n_{0}+n_{1}$ and $\phi:=\Phi(B_{\delta}(x))^{-1}\Phi$.

In order to extend \eqref{MC} to all $y\in S$, observe that for any $y$, we can pick $\epsilon(y)>0$, $n_{2}(y)\ge 1$ and compact $B^{y}\subset GL(d,\R)$ such that 
$$ \inf_{z\in B_{\epsilon(y)}(y)}P_{B^{y}}^{n_{2}(y)}(z,B_{\delta}(x))>0. $$ 
By compactness, $S=\bigcup_{i=1}^{k}B_{\epsilon(y_i)}(y_i)$ for suitable $y_1,...,y_{k}$, and a straightforward argument then shows that
$\inf_{y\in S}P_{B_2}^{n_{2}}(y,B_{\delta}(x)) > 0$
for a suitable $n_{2}\ge \max_{i=1,...,k}n_{2}(y_{i})$ and with $B_2:=\bigcup_{i=1}^{k}B^{y_{i}}$. It is now readily seen with the help of property (i) that
\begin{enumerate}
\item[(iv)] $\xi := \inf_{y\in S}P_{B}^{n_{1}+n_{2}}(y,U) > 0$, where $B:=B_{1}\cdot B_{2}$.
\end{enumerate}
By estimating $P_{C}^{n_{0}+n_{1}+n_{2}}$ with $C:=B\cdot \overline{B_{c}(\Gamma_{0})}$ following \eqref{minorization estimate} and utilizing
(iv) instead of (i), we finally obtain \eqref{MC} for all $y\in S$ (with the same $\phi$ and $m=n_{0}+n_{1}+n_{2}$). Further details can be omitted. 
\end{proof}

\begin{remark}\label{rem:BMC}
It is useful for later purposes (see Lemma \ref{M_sigma:bounded}) to point out that \eqref{MC} is ''embedded'' in a bivariate condition, obtained via consideration
of the bivariate extensions $H^{n}(x,\cdot):=\Prob(((x\Pi_{n})^{\sim},\Pi_{n})\in\cdot)$ and $H_{C}^{n}(x,\cdot):=\Prob(((x\Pi_{n})^{\sim},\Pi_{n})\in\cdot,$ $\Pi_{n}\in C)$ of $P^{n}(x,\cdot)$ and $P_{C}^{n}(x,\cdot)$, respectively. 

Keeping notation and settings from above and with $\gamma_{1}:=|\det(\Gamma_{0})|^{-d}\max_{\mathfrak{m}\in B}|\det(\mathfrak{m})|^d$ finite, a similar estimation as in \eqref{minorization estimate}
leads to
\begin{align*}
&H_{C}^{n_{0}+n_{1}+n_{2}}(y,D) \\
%&\quad\ge \int_{U\times B}\int_{B_{2\delta}(\Gamma_{0})} \1[D]((u\mathfrak{m}_{1}\mathfrak{m}_{0})^\sim,\mathfrak{m}_{1}\mathfrak{m}_{0})\,\Prob(\Pi_{n_{0}}\in d\mathfrak{m}_{0})\,H_{B}^{n_{1}}(y,du,d\mathfrak{m}_{1})  \\
&\quad\ge \gamma_{0}\int_{U\times B}\int_{B_{c}(\Gamma_{0})} \1[D]((u\mathfrak{m}_{0})^\sim,\mathfrak{m}_{1}\mathfrak{m}_{0})\,\llam(d\mathfrak{m}_{0})\,H_{B}^{n_{1}+n_{2}}(y,d(u,\mathfrak{m}_{1}))   \\
&\quad= \gamma_{0}\int_{U\times B}\int_{B_{c}(\Gamma_{0})} \1[D]((xF_{u}\Gamma_{0}^{-1}\mathfrak{m}_{0})^\sim,\mathfrak{m}_{1}\mathfrak{m}_{0})\,\llam(d\mathfrak{m}_{0})\,H_{B}^{n_{1}+n_{2}}(y,d(u,\mathfrak{m}_{1}))   \\
&\quad\ge \frac{\gamma_0}{\gamma_{1}}\,\int_{U\times B}\int_{B_{\varsigma}(I)} \1[D]((x\mathfrak{m}_{0})^\sim,\mathfrak{m}_{1}\Gamma_{0}F_{u}^{-1}\mathfrak{m}_{0})\,\llam(d\mathfrak{m}_{0})\,H_{B}^{n_{1}+n_{2}}(y,d(u,\mathfrak{m}_{1}))
%&\quad\ge \frac{\gamma_0\xi\,|\det(\Gamma_{0})|}{\gamma_{1}} \,\int_{B_{\delta}(I)}\1[D]((x\mathfrak{m}_{0})^\sim,\Gamma_{0}F_{u}^{-1}\mathfrak{m}_{0})\,\llam(d\mathfrak{m}_{0})
\end{align*}
and thus to the bivariate minorization condition
\begin{equation}\label{BMC}
H_{C}^{n_{0}+n_{1}+n_{2}}(y,\cdot)\ge q\,\psi(y,\cdot)\tag{BMC}
\end{equation} 
for all $y\in S$, some $q>0$ and a probability kernel $\psi(y,\cdot)$ on $S\times C$. It contains \eqref{MC} as a special case, for $P_C^{n_0+n_1+n_2}(y,\cdot)=H^{n_{0}+n_{1}+n_2}(y,\cdot\times C)$ and $\psi(\cdot\times C)=\phi$. 
%This will be of importance in the proof of Lemma \ref{M_sigma:bounded} because it allows us there to define regeneration epochs $\sigma_{n}$ in such a way that the matrices $M_{\sigma_{n}}$ are chosen from the \emph{compact} set $C$. See Subsection \ref{subsect:Harris chains} and the afore-mentioned lemma for further details.
\end{remark}

\section{The stopped RDE and geometric sampling}
\label{stopped RDE}

Geometric sampling and, more generally, the use of stopping times for $(M_{n},Q_{n})_{n\ge 1}$ provides a useful technique in our subsequent analysis and is thus briefly discussed next.

%We start with a nice technique, called geometric sampling, to reduce to the case where $n_0=n=1$ in assumptions \eqref{A4} and \eqref{A5}. First we prove in a more general setting, that we can consider stopped versions of our initial stochastic fixed point equation (SFPE) \eqref{SFPE}.

\subsection{R remains solution to the stopped equation}\label{subsect:stopped eq}

Let $(\mathcal{G}_{n})_{n\ge 0}$ be a filtration such that $(M_{n},Q_{n})_{n\ge 1}$ is adapted to it and $(M_{k},Q_{k})_{k>n}$ is independent of $\mathcal{G}_{n}$ for any $n\ge 0$. Consider any a.s.\ finite stopping time $\tau$ with respect to $(\mathcal{G}_{n})_{n\ge 0}$ which, by suitable choice of the latter, includes the case
that $\tau$ and $(M_{n},Q_{n})_{n\ge 1}$ are independent (pure randomization). Then it is readily checked that $R$ defined in \eqref{solution to RDE} satisfies
\begin{equation}\label{eq:R:iterate}
R=\Pi_{\tau}R^{\tau}+Q^{\tau}
\end{equation}
where
\begin{align*}
Q^n := \sum_{k=1}^n \Pi_{k-1} Q_k\quad\text{and}\quad R^{n}:=\sum_{k>n} \left(\prod_{j=n+1}^{k-1}M_{j}\right)Q_k
\end{align*}
for $n\ge 1$. But since $(M_{\tau+n},Q_{\tau+n})_{n\ge 1}$ is a copy of $(M_{n},Q_{n})_{n\ge 1}$ and also independent of $(M_{n},Q_{n})_{1\le n\le\tau}$ and $\tau$, it follows that $R^{\tau}$ is independent of $(\Pi_{\tau},Q^{\tau})$ with $R^{\tau}\stackrel{d}{=}R$. In other words, (the law of) $R$ also solves the stopped SFPE
\begin{equation} \label{SFPE:stopped}
Y\stackrel{d}{=}\Pi_{\tau}Y+Q^{\tau}
\end{equation}
and provides a stationary distribution to the  RDE
\begin{equation}\label{RDE:stopped}
Y_{n}=M_{n}'Y_{n-1}+Q_{n}',\quad n\ge 1,
\end{equation}
where $(M_{n}',Q_{n}')_{n\ge 1}$ is a sequence of i.i.d.\ copies of $(\Pi_{\tau},Q^{\tau})$. Uniqueness follows if (A1), (A2) persist to hold for the ''stopped pair'' $(\Pi_{\tau},Q^{\tau})$ together with
\begin{equation*}
\lim_{n \to \infty} \frac{1}{n} \log \norm{\Pi_{\sigma_n}} < 0 \ \Pfs
\end{equation*}
where $(\sigma_{n})_{n\ge 0}$ denotes a zero-delayed renewal process such that $\sigma_{1}=\tau$ and 
$$ \left(\sigma_{n}-\sigma_{n-1},(M_{k},Q_{k})_{\sigma_{n-1}<k\le\sigma_{n}}\right),\quad n\ge 1 $$
are i.i.d. For stopping times $\tau$ with finite mean this is indeed easily verified and we state the result (without proof) in the following lemma.

\begin{lemma} \label{lem:stopped:RDE}
The law of $R$ forms the unique solution to the SFPE \eqref{SFPE:stopped} whenever $\E{\tau} < \infty$.
\end{lemma}

In order for finding the tail behavior of $R$, we are now allowed to do so within the framework of any stopped SFPE \eqref{SFPE:stopped} with finite mean $\tau$. The idea is to pick $\tau$ in such a way that $(\Pi_{\tau},Q^{\tau})$ has nice additional properties compared to $(M,Q)$. Geometric sampling provides a typical example that will be used hereafter and therefore discussed next. Another use of this technique appears in Section \ref{sect:positive}.

\subsection{Geometric sampling}

Suppose now that $(\sigma_{n})_{n\ge 0}$ is independent of $(M_{n},Q_{n})_{n\ge 1}$ with geometric(1/2) increments, that is $\P{\tau = n}=1/2^n$ for each $n \geq 1$. Then not only Lemma \ref{lem:stopped:RDE} holds true but also the following result:

\begin{lemma} \label{lem:geomsampl1}
If $(M,Q)$ satisfies the assumption of Theorem \ref{theorem:main} and thus \eqref{MC}, then so does $(\Pi_{\tau}, Q^{\tau})$ with $n=n_0=m=1$ in \eqref{A4}, \eqref{A5} and \eqref{MC}. Also, $\lim_{n\to\infty}n^{-1}\log\,\E\|\Pi_{\sigma_{n}}\|^{\kappa}=0$ implies $\lim_{n\to\infty}n^{-1}\log\,\E\|\Pi_{n}\|^{\kappa}=0$, i.e.\ \eqref{theorem:main:assertion1}.
\end{lemma}

\begin{proof}
That \eqref{logmom M}, \eqref{logmom Q} and $\lim_{n \to \infty}n^{-1} \log \norm{\Pi_{\sigma_n}}< 0$ $\Pfs$ persist to hold under any finite mean stopping time $\tau$ has already been pointed out before Lemma \ref{lem:stopped:RDE}. As for \eqref{A3} to \eqref{A5}, we just note that  $\P{\Pi_{\tau} \in \cdot}= \sum_{k\ge 1} 2^{-n} \P{\Pi_n \in \cdot}$. Assumption \eqref{A6} ensures that the law of $R$ is nondegenerate. But since $R$ is also the unique solution to \eqref{SFPE:stopped},
\eqref{A6} must hold for $(\Pi_{\tau}, Q^{\tau})$ as well. Moreover,
\begin{align*}
\E\inf_{x\in S}|x\Pi_{\tau}| &= \sum_{n\ge 1}2^{-n}\E\inf_{x\in S}|x\Pi_{n}|
\end{align*}
in combination with
\begin{align}
\E\inf_{x\in S}|x\Pi_{n}|^{\kappa_{0}}&= \E\left(\inf_{x\in S}|(x M_{1}\cdot...\cdot M_{n-1})^{\sim}M_{n}|^{\kappa_{0}}\cdot |x M_{1}\cdot...\cdot M_{n-1}|^{\kappa_{0}}\right) \nonumber \\
&\ge \E\left(\inf_{x\in S}|x M_{n}|^{\kappa_{0}}\cdot\inf_{x\in S} |x\Pi_{n-1}|^{\kappa_{0}}\right) \nonumber \\
&= \E\inf_{x\in S}|x M_{n}|^{\kappa_{0}}\,\E\inf_{x\in S} |x\Pi_{n-1}|^{\kappa_{0}} \nonumber\\
&=...= \Big(\E\inf_{x\in S}|x M|^{\kappa_{0}}\Big)^{n} \ge 1 \label{A7:iterated}
\end{align}
for each $n\ge 1$ shows the first assertion of \eqref{A7} for $(\Pi_{\tau},Q^{\tau})$.
The remaining two moment assertions are again easily verified by standard estimates.
We therefore omit further details. Finally, suppose that $\lim_{n\to\infty}n^{-1}\log\,\E\|\Pi_{\sigma_{n}}\|^{\kappa}=0$ 
% and thus particularly $\E\|M\|^{\kappa}<\infty$, for
% $ \Prob(\tau=1)\,\E\|M\|^{\kappa}\le\E\|\Pi_{\tau}\|^{\kappa}<\infty. $
By subadditivity, $\xi:=\lim_{n\to\infty}n^{-1}\log\E\|\Pi_{n}\|^{\kappa} = \inf_{n \ge 1}n^{-1}\log\E\|\Pi_{n}\|^{\kappa} $ exists in $[-\infty,\infty)$. Since
\begin{align*}
\frac{1}{n}\log\E\|\Pi_{\sigma_{n}}\|^{\kappa} = \frac{1}{n}\log\sum_{k\ge n}\E\|\Pi_{k}\|^{\kappa}\,\Prob(\sigma_{n}=k) \ge \frac{1}{n} \sum_{k\ge n} \Prob(\sigma_{n}=k) \, \log \E\|\Pi_{k}\|^{\kappa}
\end{align*}
it is not difficult to see that $\xi>-\infty$. But then we further infer for any $\epsilon>0$ and all sufficiently large $n$ that
\begin{align*}
\frac{1}{n}\log\E\|\Pi_{\sigma_{n}}\|^{\kappa} \ge \frac{1}{n}\log\sum_{k\ge n}e^{k(\xi-\epsilon)}\,\Prob(\sigma_{n}=k) = \log\E e^{(\xi-\epsilon)\tau}
\end{align*}
and thus $\xi\le 0$ upon taking $n\to\infty$ and then $\epsilon\to 0$. By doing the same in the reverse inequality
\begin{align*}
\frac{1}{n}\log\E\|\Pi_{\sigma_{n}}\|^{\kappa} \le \frac{1}{n}\log\sum_{k\ge n}e^{k(\xi+\epsilon)}\,\Prob(\sigma_{n}=k) = \log\E e^{(\xi+\epsilon)\tau}
\end{align*}
finally shows $\xi=0$ as claimed in \eqref{theorem:main:assertion1}.
\end{proof}

\begin{remark}\label{rem:pi:small}
We note for later purposes that for all $\epsilon, \delta >0$, $x \in S$ 
$$ \Prob\left( \norm{\Pi_\tau} < \epsilon,\esl{x\Pi_\tau} \in B_\delta(x) \right) >0 ,$$
because both,
$\{ \limsup_{n \to \infty} \norm{\Pi_n} < \epsilon \}$ and $\{ \esl{x \Pi_n} \in B_\delta(x) \text{ i.o.} \}$,
are sets of probability one.
\end{remark}

In view of the previous lemma we can now make the \textbf{standing assumption} that
\begin{equation}\label{SA}
\text{If \eqref{A4},\eqref{A5},\eqref{MC} and \eqref{BMC} hold, they hold with $n_0=n=m=1$.}\tag{SA}
\end{equation}

\section{Harris recurrence and Markov renewal theory}\label{sect:MRT}
\subsection{Strongly aperiodic Harris chains}\label{subsect:Harris chains}

Here and in the following subsection let $S$ be a general separable metric space with Borel-$\sigma$-algebra $\S$. A Markov chain $(X_n)_{n \geq 0}$ on $S$ is called \emph{strongly aperiodic Harris chain}, if there exists a set $\mathfrak{R} \in \S$, called \emph{regeneration set}, such that $\P[x]{X_n \in \mathfrak{R} \text{ infinitely often}} = 1$ for all $x \in S$ (recurrence) and, furthermore,
\begin{equation}\label{Harris condition}
\inf_{x\in\mathfrak{R}}\Pkern(x, \cdot) \geq p\,\phi
\end{equation}
for some $p>0$, $r\in\N$ and a probability measure $\phi$ with $\phi(\mathfrak{R})=1$. Strong aperiodicity refers to the fact that $P$ and not $P^{m}$ for some $m\ge 2$ satsifies \eqref{Harris condition}. If $S$ itself is regenerative then $(X_n)_{n \geq 0}$ is called \emph{Doeblin chain}.  A strongly aperiodic Harris chain $(X_n)_{n \geq 0}$ possesses a nice regenerative structure as shown by the following regeneration lemma due to Athreya and Ney \cite{Athreya1978}. %(see \cite{Als1991} for a discussion???????).

\begin{lemma}\label{regenerationlemma}
On a possibly enlarged probability space, one can redefine $(X_n)_{n \geq 0}$ together with an increasing sequence $(\sigma_n)_{n \geq 0}$ of random epochs such that the following conditions are fulfilled under any $\Prob_x$, $x\in S$:
\begin{itemize}
\item[(R1)] There is a filtration $\mathcal{G}=(\mathcal{G}_n)_{n \geq 0}$ such that $(X_n)_{n \geq 0}$ is Markov adapted and each $\sigma_n$ a stopping time with respect to $\mathcal{G}$. %With respect to the canonical filtration $(\F_n)_{n \geq 0}$ of $(X_n)_{n \geq 0}$, each $\sigma_n$ is a randomized stopping time.
\item[(R2)] $(\sigma_{n}-\sigma_{1})_{n\ge 1}$ forms a zero-delayed renewal sequence with increment distribution $\P[\phi]{{\sigma_1} \in \cdot}$ and is independent of $\sigma_{1}$.  \label{R2}
\item[(R3)] For each $k\ge 1$, the sequence $(X_{\sigma_k+n})_{n \geq 0}$ is independent of $(X_j)_{0 \leq j \leq \sigma_k -1}$ with  distribution $\P[\phi]{(X_n)_{n \geq 0} \in \cdot}$.  \label{R3}
\end{itemize}
\end{lemma}

The $\sigma_{n}$, called \emph{regeneration epochs}, are obtained by the following coin-tossing procedure: If $\tau_{n}$, $n\ge 1$, denote the successive return times of the chain to $\mathfrak{R}$, then at each such $\tau_{n}$ a $p$-coin is tossed. If head comes up, then $X_{\tau_n + 1}$ is generated according to $\phi$, while it is generated according to $(1-p)^{-1}(P(X_{\nu_{n}},\cdot)-p\phi)$ otherwise. Hence, the $\sigma_{n}-1$ are those return epochs at which the coin toss produces a head. More formally, this is realized by introducing
i.i.d. Bernoulli($p$) variables $J_{0},J_{1},...$ with the following properties:  
\textit{
\begin{itemize}
\item[(R4)] For each $n \geq 0$, $J_n$ is independent of $\sigma((X_k)_{0 \le k \le n})$.
\vspace{.08cm}
\item[(R5)] $\sigma_0 :=0$ and $\sigma_n := \inf \{k > \sigma_{n-1} : \ X_{k-1} \in \mathfrak{R}, \ J_{k-1}=1 \}$ for $n\ge 1$.
\end{itemize}}
Note that $(X_{n},J_{n})_{n\ge 0}$, called \emph{split chain} (see \cite{Num1978}), is also a strongly aperiodic Harris chain with state space $S\times\{0,1\}$. Naturally, it depends on the choice of the regeneration set $\mathfrak{R}$.

%This happens with chance $p$, \emph{independent} of $\F_{\tau_n}$ \footnote{Note that regeneration at time $\tau_{n+1}$ is of course not independent of $\F_{\tau_n+1}$, since regeneration says that $X_{\tau_{n}+1}$ is distributed according to $\phi$. This is the reason for regarding the pre-regeneration times $\nu_k-1$, since regeneration in the next step happens with chance $p$.}. See also the remark after (3.1) in \cite{Athreya1978}. So  we have for all $k \geq 1$
%\[ \P{\nu_k -1 = \tau _n \text{ for some $n\geq 1$} }= p .\]

%\begin{remark} \label{anm:phi:von:R:ist:1}
%The condition $\phi(\mathfrak{R})=1$ is no restriction, if $r$ may vary: Let $\phi '$ be any probability measure $\phi '$ on $S$ satisfying the minorization condition. Since $\mathfrak{R}$ is recurrent, there is $s \in  \N$ with $\gamma = \P[\phi ']{X_s \in \mathfrak{R}}>0$. With 
%\[ \phi := \gamma^{-1} \P[\phi ']{X_s \in \cdot \cap \mathfrak{R}} \]
%one has $\phi(\mathfrak{R})=1$, and for all $x \in  \mathfrak{R}$, $A \in \S$,
%\[ \Pkern[r+s](x, A) = \int_S \Pkern[s](y, A) \Pkern[r](x,dy) \geq p \int_S \Pkern[s](y, A) \phi ' (dy) \geq p \gamma \phi(A).\]
%\end{remark}

\subsection{Markov renewal theory}

Let $(X_n, U_n)_{n \geq 0}$ be a temporally homogeneous Markov chain on $S \times \R$ such that
\[ \P{(X_{n+1}, U_{n+1}) \in A \times B|X_n, U_n} = \Pkern(X_n, A \times B)\quad\text{a.s.} \]
for all $n \geq 0$ and a transition kernel $P$. Then the associated sequence $(X_n, V_n)_{n \geq 0}$ with $V_n= V_{n-1}+U_n$ for $n\ge 1$ is also a Markov chain and called \emph{Markov Random Walk (MRW) with driving chain $(X_{n})_{n\ge 0}$}. This extends the notion of classical random walk with i.i.d.\ increments because, conditioned on $(X_n)_{n \geq 0}$, the $U_n$ are independent, but no longer identically distributed. In fact, the conditional distribution of $U_{n}$ given $(X_{k})_{k\ge 0}$ is of the form $Q((X_{n-1},X_{n}),\cdot)$ for each $n\ge 1$ and a suitable stochastic kernel $Q$. The MRW is called \emph{d-arithmetic}, if there exists a minimal $d>0$ and a measurable function $\gamma : S \to [0,d)$ such that
\[ \P{ U_1 - \gamma(x) + \gamma(y) \in d\Z|X_0= x, X_1 = y} = 1  \]
for ${P}_\pi((X_0, X_1)\in\cdot)$ almost all $(x,y) \in S^2$, and \emph{nonarithmetic} otherwise. As usual, for any distribution $\lambda$ on $S$, $\Prob_\lambda$ means $\Prob_\lambda((X_0,V_0)\in\cdot)=\lambda \otimes \delta_0$. The Markov renewal measure $\sum_{n\ge 0}\Prob_\lambda((X_{n},V_{n})\in\cdot)$ associated with the given MRW under $\Prob_{\lambda}$ is denoted as $\Bbb{U}_{\lambda}$.

Being enough for our purposes, we focus hereafter on the case when the driving chain 
is a strongly aperiodic Harris chain on compact state space and thus having a unique stationary distribution, denoted as $\pi$. 

Defining the first exit time $N(t):= \inf\{ n \geq 0 : V_n > t\}$
consider the residual lifetime process $R_t := (V_{N(t)} - t)\1[\{ N(t) < \infty \}]$ and the jump process $Z(t):= X_{N(t)}\1[\{ N(t) < \infty \}]$. A measurable function $g : S \times \R \to \R$ is called \emph{$\pi$-directly Riemann integrable} if
\begin{align}
&g(x, \cdot) \text{ is $\lambda$-a.e.\ continuous for $\pi$-almost all } x \in S  \label{dRi1}\\
\text{and}\quad&\int_S \sum_{n \in \Z} \sup_{t \in [n \delta, (n+1) \delta)} \abs{g(x,t)} \pi(dx) < \infty \ \text{ for some } \delta > 0, \label{dRi2}
\end{align}
where $\lambda$ denotes Lebesgue measure on $\R$.
The following Markov renewal theorem (MRT) is the main result of \cite{Als1997}:

\begin{theorem}\label{MRT}
%{\rm [MRT for positive Harris driving chains]}
Let $(X_n, V_n)_{n \geq 0}$ be a nonarithmetic MRW with strongly aperiodic Harris driving chain $(X_n)_{n \geq 0}$ with stationary distribution $\pi$. Let $\alpha := \E_\pi {V_1} >0$. If $g : S \times \R \to \R$ is $\pi$-directly Riemann integrable then, for $\pi$-almost all $x \in S$,
\begin{equation}\label{MRT:limit 1}
g *\Bbb{U}_x(t) := \E_x \left( \sum_{n \geq 0} g(X_n, t - V_n) \right)\to\frac{1}{\alpha} \int_S \int_{\R} g(u,v)\,dv\,\pi(du).
\end{equation}
as $t\to\infty$. Moreover, if $f : S \times (0, \infty) \to (0, \infty)$ is bounded and continuous, then 
\begin{equation}\label{MRT:limit 2}
\lim_{t \to \infty} \Erw[x]{ f(Z(t), R(t)) \1[{\{ N(t) < \infty \}}]}= L(f) 
\end{equation}
for $\pi$-almost all $x \in S$ and some constant $L(f)>0$.
\end{theorem}

\begin{remark}\label{crucial extension MRT}
The following extension of the above result follows directly upon inspection of the coupling proof given in \cite[Section 7]{Als1997}: If $\phi$ is any minorizing distribution for the transition kernel of the Harris driving chain $(X_{n})_{n\ge 0}$, then $g *\Bbb{U}_\phi(t)$ is a bounded function and converges to the limit given in \eqref{MRT:limit 1}. This fact will be used in Section \ref{sect:remove null set}.
\end{remark}

\begin{remark}\label{slln:harris}
Note that $(V_{n})_{n\ge 0}$ satisfies the strong law of large numbers, viz.
\begin{equation} \lim_{n \to \infty} \frac{V_n}{n}= \alpha\quad\Pfs\label{Pfs:Konvergenz:Vn}\end{equation}
The number $\alpha= \E_\pi {V_1}$ is called the \emph{drift} of $(X_{n},V_{n})_{n\ge 0}$.
\end{remark}

\section{Measure change and tail behaviour of $\sup_{n \geq 1}|x \Pi_n|$} \label{sect:tailMat}

Returning to the model described in the Introduction, we proceed with a short account of the ideas in \cite{Kesten1973} and \cite{LePage1983}. %A detailed discussion may be found in \cite{Mentemeier2009}. Let $S:=S$ be the unit sphere of $d$-row vectors with euclidean norm equal to 1, $\S$ the Borel-$\sigma$-algebra on $S$. 
Recall that $S=S^{d-1}$ and define 
\begin{align*}
 X_n  := (X_0 \Pi_n)^\sim\quad\text{and}\quad
U_n  := \log |X_{n-1} M_n|  .
\end{align*}
Since $U_n=  \log |X_0 \Pi_n| - \log |X_0 \Pi_{n-1}|$, the sequence $(X_n, V_n)_{n \geq 0}$ forms a MRW on $S \times \R$ with initial values $(X_0, V_0)$, and 
\begin{equation*}
\{ N(t) < \infty \} = \left\{ \sup_{n \geq 1} \log |X_0 \Pi_n| > t\right\} .
\end{equation*}

Under the assumptions of Theorem \ref{theorem:main} (with $n_{0}=n=1$ in \eqref{A4} and \eqref{A5}), $(X_{n})_{n\ge 0}$ is easily seen to be a strongly aperiodic Harris chain with transition kernel $P=P^{1}$ defined in Section \ref{section:MC}.
However, $(X_n, V_n)_{n \geq 0}$ does not satisfy the conditions of the MRT, since by \eqref{Pfs:Konvergenz:Vn}, 
\begin{equation*}
\alpha=\lim_{n\to\infty}\frac{\log |X_{0}\Pi_{n}|}{n}\le\lim_{n\to\infty}\frac{\log\|\Pi_{n}\|}{n}=\beta<0\quad\Pfs
\end{equation*}
A MRW with positive drift is indeed obtained after a change of measure (harmonic transform) for which it is crucial that $\P{\log \norm{M} > 0}>0$ which in turn follows from assumption \eqref{A7}.

\begin{theorem} \label{thm:transforms:exist}
Under the assumptions of Theorem \ref{theorem:main}, there exist $\kappa \in (0, \k[0]]$ and a positive continuous function $r : S \to (0, \infty)$ such that
\begin{align}
&\Ekappa[x]{f(X_0,V_0,X_1, V_1, \dots, X_n, V_n)}\nonumber \\ 
&\hspace{1cm}:= \frac{1}{r(x)} \E_{x}\Big(r(X_n) e^{\kappa V_n} f \left(X_0,V_0,X_1, V_1, \dots, X_n, V_n \right)\Big), \label{Def:transformiertesMass}
\end{align}
for all bounded continuous functions $f$ and all $n \geq 0$, defines a distribution  $\Probk_x$ for each $x \in S$. Under $\Probk_x$, $(X_n, V_n)_{n \geq 0}$ is a MRW with positive drift and satisfies the assumptions of Theorem \ref{MRT}. The constant $\kappa$ is the unique value such that \eqref{theorem:main:assertion1} holds true, i.e.
$$ \lim_{n \to \infty} n^{-1} \log \E{\norm{\Pi_n}^\kappa}=0 .$$
\end{theorem}

Note that, for each $x\in S$, $\Prob_{x}$ and $\Pkappa[x]$ are equivalent probability measures on any $\sigma((X_{k},V_{k}):k\le n)$, $n\ge 0$. We often write $\Pfs$ and $\Pkfs$ as shorthand for $\Prob_x$-a.s.\ and $\Pkappa[x]{}$-a.s.\ for all $x \in S$, respectively. Moreover probabilities under $\Prob$ without subscript are always understood as being independent of the initial state und thus the same under any $\Prob_{x}$.

\subsection{Proof of Theorem \ref{thm:transforms:exist}: Choice of $\kappa$ and $r$}

 Defining the positive operators
\[ T_{\varkappa}: \mathcal{C}(S) \to \mathcal{C}(S), \ \  f(x) \mapsto \Erw{e^{\varkappa \log |x M| }f(\esl{xM})},\quad\varkappa\in (0,\kappa_{0}], \]
we must find $\kappa$ such that $T_{\kappa}$ has maximal eigenvalue 1 with positive eigenfunction $r$. Due to our standing assumption, $T_{\varkappa}$ is even strictly positive in the sense that $T_{\varkappa}f$ is everywhere positive whenever $0\ne f\ge 0$.
Indeed, for any such $f$, the set $U_{f}=\{f>0\}$ is nonempty and open by continuity whence, using \eqref{A4} with $n=1$, we infer
$$ T_{\varkappa}f(x) \ge \int |y|^{\varkappa}\,\1[U_{f}](y^{\sim})\,f(y^{\sim})
\ \Prob(xM \in dy) > 0 $$
for all $x\in S$. The strict positivity will enable us to provide an elegant proof of the important Lemma \ref{lem:r:symmetric} below.

\begin{lemma}\label{lem:spectral radius}
Let $\rho(\varkappa)$ be the spectral radius of $T_\varkappa$, i.e. $ \rho( \varkappa ) = \lim_{n \to \infty} \norm{T_\varkappa^n}^{1/n}$. Then $T_\varkappa$ has an eigenvalue of maximal modulus equal to $\rho(\varkappa)$.
\end{lemma}

\begin{proof}
The adjoint operator $T_\varkappa^* : \mathcal{C}(S)^* \to \mathcal{C}(S)^*$, $\mathcal{C}(S)^*$ being the space of regular bounded signed measures on $S$, is weakly compact, i.e.\ it maps bounded sets to weakly sequentially compact sets. This follows by Prokhorov's theorem because
\[ \norm{T_\varkappa^*}=\norm{T_\varkappa} \leq \E{\norm{M}^\varkappa} < \infty \]
and $S$ is compact. By \cite[Theorem VI.4.8]{Dunford1958}, $T_\varkappa$ is then weakly compact as well, and by \cite[Corollary VI.7.5]{Dunford1958}, $T_\varkappa^2$ is compact. Hence, by \cite[Lemmata VII.4.5 \& 6]{Dunford1958}, the spectrum of $T_\varkappa$ is pure point (maybe except for $0$) and $T_\varkappa$  possesses an eigenvalue $\lambda_{\varkappa}$ that is maximal in modulus, i.e. $|\lambda_{\varkappa}|=\rho(\varkappa)$.
\end{proof}

The following argument shows the existence of $\kappa \in (0, \kappa_0]$ with $\rho(\kappa)=1$: 
As one can readily verify by induction, $T_{\varkappa}^n f(x) = \Erw{e^{\varkappa \log \abs{x \Pi_n}}f(\esl{x\Pi_n})}$, and we infer $\rho(\kappa_0) \geq 1$ upon choosing $f =\1[S]$ and using \eqref{A7:iterated}. If $\rho(\kappa_0)=1$ we are done, so suppose that $\rho(\kappa_0)>1$ and thus 
$\|T_{\kappa_0}^n\| >1$ 
for all sufficiently large $n$. 

\vspace{.06cm}
Since $\varkappa \mapsto \norm{T_\varkappa^n f}$ is log-convex and thus continuous on $(0,\kappa_{0})$ and lower semicontinuous on $(0, \kappa_0]$ for each $f \in \mathcal{C}(S)$ and $n \geq 1$ (use H\"older's inequality), the same holds true for $\varkappa \mapsto \norm{T_\varkappa^n}$ as its pointwise supremum. It follows that $\norm{T_{\kappa_1}^n} > 1$ for some $\kappa_1 \in (0, \kappa_0)$ and all sufficiently large $n$ and therefore $\varrho(\kappa_{1})\ge 1$. But we also have $\rho(\kappa_2)<1$ for some $\kappa_2\in (0, \kappa_0)$ because $\beta < 0$ (and by the Furstenberg-Kesten theorem). 

Finally, again as pointwise limit of the log-convex functions $\varkappa \mapsto \norm{T_\varkappa^n}^{1/n}$, $\rho(\varkappa)$ is  log-convex and thus continuous on $(0, \kappa_0)$. Hence, $\rho(\kappa)=1$ for some unique $\kappa\in (0,\kappa_{0})$. That $\kappa$ also satisfies \eqref{theorem:main:assertion1} follows from the following more general lemma.

\begin{lemma} \label{lem:property (4)}
For each $\varkappa\in (0,\kappa_{0}]$,
$$ \rho( \varkappa ) = \lim_{n \to \infty} (\E{\norm{\Pi_n}^\varkappa})^{1/n}. $$
\end{lemma}

\begin{proof}
Obviously, $$\rho(\varkappa)=\lim_{n\to\infty}\sup_{x\in S}(\E{\abs{x\Pi_{n}}^\varkappa})^{1/n}\le\liminf_{n \to \infty} (\E{\norm{\Pi_n}^\varkappa})^{1/n}.$$ For the converse note that, by \cite[Prop.\ 3.2]{BouLac1985},
$Z_{x_0}:=\inf_{n\ge 0}\norm{\Pi_{n}}^{-1}\abs{x_0\Pi_{n}}>0$ a.s.\ 
for any $x_0\in S$, whence
\begin{equation*}
\sup_{x\in S}\E{\abs{x\Pi_{n}}^{\varkappa}}\ \ge\ \E{\norm{\Pi_n}^{\varkappa}}\,\frac{\E{\abs{x_{0}\Pi_{n}}^{\varkappa}}}{\E{\norm{\Pi_n}^\varkappa}}\ \ge\ \E{\norm{\Pi_n}^{\varkappa}}\,\frac{\E{Z_{x_{0}}\norm{\Pi_n}^\varkappa}}{\E{\norm{\Pi_n}^\varkappa}}
\end{equation*}
and therefore (using Jensen's inequality)
\begin{equation*}
\rho( \varkappa )\ \ge\ \limsup_{n\to\infty}(\E{\norm{\Pi_n}^\varkappa})^{1/n}\lim_{n\to\infty}\frac{\E{Z_{x_{0}}^{1/n}\norm{\Pi_n}^\varkappa}}{\E{\norm{\Pi_n}^\varkappa}}\ =\ \limsup_{n\to\infty}(\E{\norm{\Pi_n}^\varkappa})^{1/n}.
\end{equation*}
which completes the proof.
\end{proof}

\begin{lemma} \label{lem:r:symmetric}
Let $\rho(\kappa)=1$. Then $T_\kappa$ has maximal eigenvalue $1$ with one-dimen\-sional eigenspace containing a positive eigenfunction $r$ which further is symmetric, i.e.\ $r(x)=r(-x)$ for all $x\in S$. 
\end{lemma}

\begin{proof}
The following argument goes back to Karlin \cite[Section 5]{Karlin1959} and hinges on the strict positivity of $T_{\kappa}$.
By Lemma \ref{lem:spectral radius}, $T_\kappa$ has eigenvalue $\lambda$ with $\abs{\lambda}=1$. Let $f$ be a corresponding eigenfunction. Obviously, $T_\kappa f = \lambda f$ implies 
\[ T_\kappa \abs{f} \geq \abs{f} .\]
Suppose that $T_\kappa \abs{f}-\abs{f}\ne 0$. Then, by the strict positivity of $T_{\kappa}$, we have that $T_\kappa (T_\kappa \abs{f} - \abs{f})$ is positive and thus $>\eta$ for some $\eta>0$ chosen small such that, furthermore, $T_{\kappa}\abs{f}<1/\eta$  \makebox{($S$ compact)}. 
From this we further infer
$$ T_\kappa^2 \abs{f}-T_{\kappa}\abs{f} > \eta >\eta^{2}T_{\kappa}|f|,\quad\text{hence }T_\kappa^2\abs{f}> (1+\eta^{2})T_{\kappa}\abs{f} $$
and thereby $T_{\kappa}^{n}T_{\kappa}|f|>(1+\eta^{2})^{n}T_{\kappa}\abs{f}$ for all $n\ge 1$ upon iteration. Consequently, $\|T_{\kappa}^{n}\|>(1+\eta^{2})^{n}$ for all $n\ge 1$ and thus $\rho(\kappa)>1$, a contradiction that leads to the conclusion that $T_{\kappa}|f|=|f|$ und thus that $r:=|f|$ is a positive eigenfunction for the eigenvalue 1.

Now, suppose there is another eigenfunction $g$, linearly independent of $r$ and w.l.o.g.\ real-valued (for, if $g$ is an eigenfunction, then so are its real and imaginary parts if nontrivial). Pick $\epsilon$ such that $h:=r + \epsilon g$ is nonnegative, but $h (x)=0$ for some $x$. By linear independence, $h$ does not vanish everywhere. Since it is again an eigenfunction, the strict positivity of $T_{\kappa}$ implies that it must be positive everywhere which is a contradiction. Hence $r$ must be the unique eigenfunction modulo scalars.

Finally, we must prove the asserted symmetry of $r$. To this end note first that $T_\kappa$ maps symmetric functions to symmetric functions. Its weak compactness entails that $T_\kappa^2$ is a compact operator \cite[Corollary VI.7.5]{Dunford1958} and thus maps bounded sequences to sequences with (strongly) convergent subsequences. As a consequence, any accumulation point $g$ of the bounded sequence
$n^{-1}T_\kappa \sum_{k=1}^n T_\kappa^k \1[S]$, $n\ge 1$,
is a \emph{continuous} positive symmetric function with $Tg = g$ and thus a multiple of $r$. Hence, $r$ must be symmetric.
\end{proof}

Now
$\Phatkern[{\kappa}] f(x,t):=r(x)^{-1}\Erw{|x M|^{\kappa}\,f \bigl( \esl{xM}, t + \log |xM| \bigr) r(\esl{xM})}$
defines a Markov transition kernel on $S\times\Bbb{R}$ corresponding to $({}^{\kappa}\!\Prob_x((X_{n},V_{n})_{n\ge 0}\in\cdot))_{x \in S}$ as defined by \eqref{Def:transformiertesMass}. Its associated ''marginal''
\begin{equation}\label{def:Pkappa}
\Pkern[{\kappa}] f(x):=\frac{1}{r(x)} \Erw{|x M|^{\kappa}\,f \bigl( \esl{xM}\bigr) r(\esl{xM})}
\end{equation}
is the transition kernel of $(X_{n})_{n\ge 0}$ under $({}^{\kappa}\!\Prob_x)_{x\in S}$.

%\vspace{.2cm}
%In the following, we often write $\Pfs$ and $\Pkfs$ as shorthand for $\Prob_x$-a.s.\ and $\Pkappa[x]{}$-a.s.\ for all $x \in S$, respectively. Moreover probabilities under $\Prob$ without subscript are always understood as being independent of the initial state und thus the same under any $\Prob_{x}$.
%
%Next we show that there is some positive eigenvalue. Since the spectra of $T_\kappa$ and $T_\kappa^*$ are equal, we consider the normalized operator
%\begin{align*}
% \tilde{T}_\kappa^* : \mathfrak{M}^1(reg)  \to \mathfrak{M}^1(reg), \ \ 
%\nu  \mapsto \frac{1}{T_\kappa^* \nu(S)} T_\kappa^*\nu 
%\end{align*}
%$\tilde{T}_\kappa^*$ maps positive measures into positive measures, so we can consider it as a map on probability measures. Now by the Schauder-Tychonoff fixed point theorem \cite[V.10.5]{Dunford1958}, it has a fixed point $\nu$, and this is a eigenmeasure of $T_\kappa^*$ with eigenvalue 
%\[ \lambda=T_\kappa^* \nu (\projR) = \int_{\projR} \nu(d\rk{x}) T_\kappa \1[{\projR}](\rk{x}) = \int_{\projR} \nu(d\rk{x}) \E{\norm{xM_1}^\kappa} \]
%
%By making use of assumption \eqref{A4} and some estimates, one infers that $\lambda^n$ is of order $\E{\norm{\Pi_n}}$, so $\lambda=\rho(\kappa)$. Let $\tilde{r}$ be any eigenfunction with eigenvalue $\lambda$.
%Since $T_\kappa$ maps positive functions into positive functions, by splitting into real and imaginary parts, and again in positive and negative part,
%\[ T_\kappa \Re(\tilde{r})^+ = \lambda \Re(\tilde{r})^+, \]
%so $\rk{r}:= \Re(\tilde{r})^+ $ is the positive eigenfunction with eigenvalue 1 we have searched for.

\subsection{Proof of Theorem \ref{thm:transforms:exist}: Checking the assumptions of the MRT}
%From now on, $S:= S$, $\S$ the Borel-$\sigma$-algebra, $r(x):= \rk{r}(\rk{x})$. 
This section corresponds to \cite[Proposition 2]{Kesten1973}, but provides a much shorter proof, even if technical details not mentioned here had been included. A random variable $T\ge 0$ is called geometrically bounded if it has exponentially decreasing tails. 

\begin{lemma} \label{lem:Harrisrecurrence}
Suppose \eqref{A4}, \eqref{A5} and \eqref{SA}. Then, for each $x \in S$, there exists some $\delta >0$ such that $B_\delta(x)$ is a regeneration set with respect to $P$ as well as $\Pkern[{\kappa}]$, and the minorization condition holds with the same probability measure $\phi$ defined in Lemma \ref{lem:MC}. Moreover,  $(X_n)_{n\ge 0}$ is a strongly aperiodic Doeblin chain under $(\Prob_y)_{y\in S}$ as well as under $(\Pkappa[y])_{y\in S}$.
\end{lemma}
\begin{proof}
Let $x\in S$. By \eqref{MC} with $m=1$, we know that $P(y,\cdot)\ge P_{C}(y,\cdot)\ge p \phi$ for suitable $\delta,p,C,\phi$ and all $y\in S$ (see Lemma \ref{lem:MC}, especially (ii) in the proof for the definition of $\phi$). In particular,
$\inf_{y\in S}P(y,B_{\delta}(x))\ge p\phi(B_{\delta}(x))=p>0$ which gives geometrically bounded times to visit $B_{\delta}(x)$ under any $\Prob_{y}$ (uniformly in $y\in S$), thus $B_{\delta}(x)$ is regenerative with respect to $P$. In order to get the same with respect to $\Pkern[{\kappa}]$, we first note that
$$ \gamma_{2}:=\min_{{z_1}, {z_2} \in S} \frac{r(z_1)}{r({z_2})}\quad\text{and}\quad\gamma_{3}:=\min_{z \in S,\,\mathfrak{m}\,\in C} |z\mathfrak{m}|^\kappa $$
are clearly both positive (here the compactness of the set $C\subset GL(d,\R)$ enters in a crucial way). But then we infer % with the help of \eqref{A5} with $n_{0}=1$ (as $\delta<c$) that
\begin{align*}
\Pkern[{\kappa}](y,A)\ \ge\ \Pkern[{\kappa}]_{C}(y,A)\ 
&\ge\ \frac{1}{r(y)} \int_{C}  |y\mathfrak{m}|^\kappa \, \1[A]((y\mathfrak{m})^\sim)\,r(\esl{y\mathfrak{m}})\,\Prob(M\in d\mathfrak{m})  \\
%&\ge \left( \min_{{z_1}, {z_2} \in S} \frac{r(z_1)}{r({z_2})} \right)\left(\min_{z \in \overline{B_\delta(x)},\,\mathfrak{m}\,\in \overline{B_c(\Gamma_{0})}} |z\mathfrak{m}|^\kappa \right)\Prob((yM)^\sim\in B) \\
&\ge\ \gamma_{2}\gamma_{3}\int_{C}\1[A]((y\mathfrak{m})^\sim)\,\Prob(M\in d\mathfrak{m})\\
%&\ge \gamma_{0}\gamma_{1}\gamma_{2}\int_{B_{\delta}(\Gamma_{0})}\1[A]((y\mathfrak{m})^\sim)\,\llam(d\mathfrak{m})
&=\ \gamma_{2}\gamma_3\,\Pkern_{C}(y,A) \ge p \gamma_{2}\gamma_{3}\phi(A) 
\end{align*}
for any measurable $A\subset B_{\delta}(x)$ and $y\in S$ and thereby that $B_{\delta}(x)$ is regenerative with respect to $\Pkern[{\kappa}]$ as well. 
\end{proof}

%Let $A$ be a feasible matrix, which exists by \eqref{A6}, with dominant left eigenvector $u$. Put $S_u:=\{ x \in S: \abs{xu} \leq\epsilon \}$ for some $\epsilon >0$, $S_u^+:= S_u^c \cap \{ x \in S: xu>0 \}$, and $S_u^-:= S_u^c \cap \{ x:xu<0 \}$. Perron-Frobenius theory, which extends to feasible matrices (see \cite[Lemma 8.2.7]{HJ1985}), tells us that $\esl{xA^n}$ converges uniformly to $u$ on $S_u^+$, and to $-u$ on $S_u^-$. This gives geometrically bounded times for entering any neighbourhood of $u$ resp. $-u$. By \eqref{A4} and the continuity of $x\mapsto\Prob(\esl{xM}\in\cdot)$ in the Prokhorov metric (see Lemma \ref{Prokhorov:continuity}), there is a positive chance to enter some neighbourhood of $u$ from a sufficiently small neighbourhood of $-u$; and similarly, by compactness, to enter $S_u^+$ from $S_u$. So for some $\delta >0$, $B_\delta(u)$ is Harris recurrent. 
%But then, again by \eqref{A4} and Lemma \ref{Prokhorov:continuity}, we have geometrically bounded times to reach sufficiently small neighbourhoods of any $x \in S$ from $B_\delta(u)$. Hence, for each $x \in S$, there is some $\delta >0$ such that $B_\delta(x)$ is Harris recurrent.

\emph{The stationary distribution of this chain with respect to $\Pkern[{\kappa}]$ is always denoted as $\pi$ hereafter.} Note that also $(X_n, U_n)_{n \geq 0}$ is a stationary sequence under $(\Pkappa[x])_{x \in S}$.

\vspace{.2cm}
We will need further information on the behavior of $\abs{xM_{\sigma_n}}$, which follows from the bivariate minorization condition stated in Remark \ref{rem:BMC} (with $n_{0}+n_{1}+n_{2}=1$ due to our standing assumption).

\begin{lemma} \label{M_sigma:bounded}
Let $B_{\delta}(x)$ be a regenerative ball with minorizing probability measure $\phi$ 
as in Lemma \ref{lem:MC}. Then we can choose a sequence of regeneration epochs $(\sigma_n)_{n \geq 0}$ such that, for suitable $\mathfrak{C} > \mathfrak{c}>0$,
\begin{equation}\label{lower bound}
\mathfrak{C}\ \geq\ \norm{M_{\sigma_n}}\ \geq\ \inf_{y \in S} \abs{yM_{\sigma_n}}\ \geq\  \mathfrak{c}\quad\Pfs \quad (\text{and thus } \Pkfs)
\end{equation}
for all $n\ge 1$. As a particular consequence, $U_{\sigma_1}=\log|X_{\sigma_{1}-1}M_{\sigma_{1}}|$ is $\Pfs$ bounded, that is taking values in some finite interval $[s_{*},s^{*}]$.
\end{lemma}

\begin{proof}
We just note that, by \eqref{BMC}, we may generate $(X_{\sigma_{n}},M_{\sigma_{n}})$ given $X_{\sigma_{n}-1}=y$ at any regeneration epoch $\sigma_{n}$ according to $\psi(y,\cdot)$ having first marginal $\phi$, thus $X_{\sigma_{n}}\stackrel{d}{=}\phi$. Moreover, $M_{\sigma_{n}}\in C$ $\Pfs$ for a compact $C\subset GL(d,\R)$ which entails $\norm{M_{\sigma_n}} \leq \mathfrak{C}$, $\|M_{\sigma_{n}}^{-1}\|\le\mathfrak{c}^{-1}$ $\Pfs$ for some constants. Since
\begin{equation*}
\inf_{x \in S} \abs{xM_{\sigma_n}} = \inf_{x \in S} \frac{\abs{xM_{\sigma_n}}}{\abs{xM_{\sigma_{n}}M_{\sigma_{n}}^{-1}}} = \frac{1}{\sup_{x\in S}|xM_{\sigma_{n}}^{-1}|} = \frac{1}{\|M_{\sigma_{n}}^{-1}\|},
\end{equation*}
we infer \eqref{lower bound}. The $\Prob$-almost sure boundedness of $U_{\sigma_{1}}$ then follows directly from$\|M_{\sigma_{1}}\|^{-1}\le e^{|U_{\sigma_{1}}|}\le\|M_{\sigma_{1}}\|$.
\end{proof}

\textbf{From now on}, we will always assume that \eqref{lower bound} is in force when given sequence of regeneration epochs $(\sigma_{n})_{n\ge 0}$. The regeneration set will always be
some ball $B_{\delta}(x)$, the i.i.d.\ coin-tossing variables are denoted by $J_n$, so $J_{n}\stackrel{d}{=}$ Bernoulli$(p)$ for $n\ge 0$, and the minorizing measure by $\phi$ as in Lemma \ref{lem:MC} (which naturally depends on $B_{\delta}(x)$).

\vspace{.2cm}
The following result will be needed in section \ref{sect:positive}:

\begin{lemma}\label{contractive}For any sequence $(\sigma_{n})_{n\ge 0}$ of regeneration epochs as described above, and all $y \in S$ 
$$ \P[y]{ \norm{\Pi_{\sigma_1}} < 1 } >0 .$$
\end{lemma}
\begin{proof}
Due to geometric sampling, in particular Remark \ref{rem:pi:small}, 
\begin{align*}
\P[y]{\norm{\Pi_{\sigma_1}} < 1}\ &\ge\ \P[y]{\norm{\Pi_{2}} < 1, \sigma_1 \le 2} \\
&\ge\ \P{\norm{\Pi_{1}} < \frac{1}{\mathfrak{C}}, \esl{x\Pi_1} \in B_\delta(x), J_1 = 1} \\
&\ge\ p\,\P{\norm{\Pi_{1}} < \frac{1}{\mathfrak{C}}, \esl{x\Pi_1} \in B_\delta(x)} > 0,
\end{align*}
where $B_{\delta}(x)$ denotes the regenerative ball for $(\sigma_{n})_{n\ge 0}$.
\end{proof}

We now turn to the lattice-type of $(X_{n},V_{n})_{n\ge 0}$, which is the same under $(\Prob_{x})_{x\in S}$ and $(\Pkappa[x])_{x\in S}$. Kesten \cite{Kesten1973} imposes an additional assumption involving so-called feasible matrices in order to ensure that $(X_{n},V_{n})_{n\ge 0}$ is nonarithmetic. But in view of assumption \eqref{A5} it should be no surprise that this is not needed here. The following lemma provides the confirmation in an even stronger form.

\begin{lemma}\label{lattice-type}
Suppose \eqref{A4}, \eqref{A5} and \eqref{SA}. Then $(X_{n},V_{n})_{n\ge 0}$ is nonarithmetic under $(\Prob_{x})_{x\in S}$, in fact
$$ \E_{x}\abs{\Erw{e^{it V_1} | X_0, X_1}} < 1 $$
for all $t\ne 0$ and $\pi$-almost all $x\in S$.
\end{lemma}

\begin{proof}
If the assertion fails to hold, there exists a distribution $\nu$ on $S$, absolutely continuous with respect to $\pi$, such that $\E_{\nu}\abs{\Erw{e^{it V_1} | X_0, X_1}}=1$ for some $t\ne 0$. As a consequence,
$$ \Erw{e^{it V_1} | X_0, X_1} = e^{it f(X_0, X_1)} \ \Pfs[x] $$
for some measurable function $f$ and $\nu$-almost all $x\in S$ or, equivalently,
\begin{equation} \label{nonarith:cond} 
\P[\nu]{ V_1 \in f(X_0, X_1) + t^{-1}\Z} = 1.
\end{equation}
W.l.o.g.\ suppose $t=1$ hereafter. Due to \eqref{A5} and \eqref{SA}, a nonzero component of $\P[x]{(X_1,V_1) \in\cdot}$ is given by 
$$ \Lambda_{x}(A\times B):=\gamma_0 \int_{B_c(\Gamma_0)} \1[A](\esl{x\m}) \1[B](\log \abs{x\m}) \llam(d\m) $$ 
for measurable $A\subset S$, $B\subset\R$ and any $x\in S$.
The mapping $\m \mapsto x\m$ induces an absolutely continuous measure on $\R^d$ with some $\llam^d$-density $g$, say. Switching to spherical coordinates, there are $\epsilon_1, \epsilon_2 >0$ such that
\begin{align*} 
\Lambda_{x}(A\times B) = \gamma_0 \int_{\abs{x \Gamma_0}- \epsilon_1}^{\abs{x \Gamma_0} + \epsilon_1} \int_{B_{\epsilon_2}(\esl{x \Gamma_0}) \cap S} \1[A](\omega) \1[B](s) g(s \omega) \sigma(d\omega) \frac{1}{s^{1+d}}\ ds
\end{align*}
where $\sigma$ is a measure on the sphere $S$. Now, if \eqref{nonarith:cond} were true with $t=1$, then
$$ \Lambda_{x}(S\times\R)\ =\ \gamma_0\int_{B_{\epsilon_2}(\esl{x \Gamma_0}) \cap S}\int_{\abs{x \Gamma_0}- \epsilon_1}^{\abs{x \Gamma_0} + \epsilon_1} \1[f(x,\omega) + \Z](s) g(s\omega) \frac{1}{s^{1+d}}\ ds\ \sigma(d\omega)\ >\ 0 $$
for all $x$ which is impossible because the inner integral over a countable set is clearly zero for any fixed $\omega$.
\end{proof}

For the proof of Theorem \ref{thm:transforms:exist}, it finally remains to verify that  $(X_{n},V_{n})_{n\ge 0}$ has positive drift under $\Pkappa$. The subsequent argument simplifies the original one given by Kesten \cite{Kesten1973}. 

\begin{lemma}\label{lem:pos drift}
Under $\Pkappa$, $(X_{n},V_{n})_{n\ge 0}$ has positive drift, given by
$$ \alpha:=\Ekappa[\pi]\left(\frac{V_{n}}{n}\right)=\frac{1}{n} \int\frac{1}{r(x)}\E |x\Pi_{n}|^{\kappa}\log |x\Pi_{n}|\,r((x\Pi_{n})^{\sim})\ \pi(dx) $$
for any $n\ge 1$.
\end{lemma}

\begin{proof}
For each $n\ge 1$, the function
$$ g_{n}(\varkappa):=\int\frac{1}{r(x)}\E |x\Pi_{n}|^{\varkappa}\,r((x\Pi_{n})^{\sim})\ \pi(dx) $$
is finite and thus convex for $\varkappa\in [0,\kappa_{0}]$. Moreover $g_{n}(\kappa)=1$ and the left derivative at $\kappa$ equals
$$ \lim_{x \uparrow \kappa} \frac{g_n(x) - g_n(\kappa)}{x-\kappa} = \alpha n .$$ 
By convexity, $\alpha$ is positive if we can show that $g_{n}(\varkappa)<1$ for some $n$ and some $\varkappa<\kappa$. To this end pick any $\varkappa<\kappa$ and recall that $\rho(\varkappa)<1$. It follows that
\begin{equation*}
g_{n}(\varkappa) = \int\frac{T_{\varkappa}^{n}r(x)}{r(x)}\,\pi(dx) = \int\frac{\|r\|_{\infty}}{r(x)}\,T_{\varkappa}^{n}\left(\frac{r(x)}{\|r\|_{\infty}}\right)\,\pi(dx)
\le C\|T_{\varkappa}^{n}\|
\end{equation*}
for some $C\in (0,\infty)$ and all $n\ge 1$. As $\|T_{\varkappa}^{n}\|^{1/n}\to\rho(\varkappa)$, we infer $g_{n}(\varkappa)\to 0$ and thus the desired result.
\end{proof}

\subsection{Tail behavior of $\sup_{n\ge 1}\abs{x \Pi_n}$}

With the help of the MRT \ref{MRT}, we are now able to prove the following result on the tail behavior of $\sup_{n\ge 1}\abs{x \Pi_n}$.

\begin{proposition}\label{tail:sup:Pin}
Under the conditions of Theorem \ref{theorem:main} and with $r$ as defined in Lemma \ref{lem:r:symmetric},
\[ \lim_{t \to \infty} t^{\kappa}\,\P{\sup_{n \geq 1} \abs{x \Pi_n} > t} = L \,r(x) ,\]
for $\pi$-almost all $x \in S$ and some $L>0$.
\end{proposition}

\begin{proof}
The function $f : S \times (0, \infty) \to (0, \infty),\ (y, s) \mapsto e^{-\kappa s}/r(y)$ 
is bounded and continuous whence, by an application of the MRT,
\begin{align*} 
L(f):=\lim_{t \to \infty} \Ek[x]{ f(Z(t),R(t)) \1[\{ N(t) < \infty \}] }
\end{align*}
exists and is positive. On the other hand, we have
\begin{align}
\Ek[x]{ f(Z(t),R(t)) \1[\{ N(t) < \infty \}] }
&=  \sum_{n\ge 1} \Ek[x]{ f(X_n, V_n -t)\, \1[\{ N(t) =n \}] }\nonumber \\
&=  \sum_{n\ge 1} \Ek[x]{ \frac{1}{r(X_n)}\, e^{-\kappa V_n + \kappa t}\, \1[\{ N(t) =n \}] } \nonumber \\
&= \frac{e^{\kappa t}}{r(x)}\sum_{n\ge 1} \Erw[x]{ \frac{1}{r(X_n)}\, e^{-\kappa V_n} \,  r(X_n)\,e^{\kappa V_n}\, \1[\{ \tau(t) =n \}] }\nonumber\\ 
%&=  \frac{e^{\kappa t}}{r(x)}  \ \sum_{n\ge 1}  \E_x{  \1[\{ \tau(t) =n \}] } \nonumber \\
&=  \frac{e^{\kappa t}}{r(x)}  \,  \Prob_{x}(N(t) < \infty)\nonumber\\ 
&=  \frac{e^{\kappa t}}{r(x)}  \,  \P{ \sup_{n \geq 1} \ \log \abs{x\Pi_n} > t}, \label{Konvergenzratenrechnung}
\end{align}
which provides the asserted result upon substituting $e^{t}$ by $t$.
\end{proof}

\subsection{Tail behavior of $\sup_{n\ge 1}\abs{x \Pi_{\sigma_{n}-1}}$}

Let $B_{\delta}(x_0)$ be any regenerative ball with associated sequences $(\sigma_{n})_{n\ge 0}$ and $(\tau_{n})_{n\ge 1}$ of regeneration epochs and hitting times, respectively.
In Section \ref{sect:positive}, we will need and therefore show below that 
$$ \limsup_{t \to \infty} t^{\kappa}\,\P[x]{\sup_{n \geq 1} \abs{x \Pi_{\sigma_{n}-1}} > t}>0 $$
for $\pi$-almost all $x \in S$. The proof hinges on the following proposition similar to Proposition \ref{tail:sup:Pin} above.

\begin{proposition}\label{tail:sup:Pitaun}
Under the conditions of Theorem \ref{theorem:main} and with $r$ defined in Lemma \ref{lem:r:symmetric}, there exists $L'>0$ such that
\[ \lim_{t \to \infty} t^{\kappa}\,\P[x]{\sup_{n \geq 1} \abs{x \Pi_{\tau_n}} > t} = L' \,r(x) ,\]
for $\pi$-almost all $y\in B_\delta(x_0)$.
\end{proposition}

Since $V_{\tau_n}= \log \abs{x \Pi_{\tau_n}}$ a.s.\ under $\Prob_x$ and $\Pkappa[x]$, we can proceed exactly as in the proof of Proposition \ref{tail:sup:Pin}, provided that the assumptions of the MRT \ref{MRT} hold for the sequence $(X_{\tau_n}, V_{\tau_n})_{n \geq 0}$ under $(\Pkappa[x])_{x\in S}$, which is verified by the subsequent lemma. We note that \eqref{Def:transformiertesMass} extends to
\begin{align}
&\Ekappa[x]{f(X_0,V_0,X_1, V_1, \dots, X_{\tau_n}, V_{\tau_n})}\nonumber \\ 
&\hspace{1cm}= \frac{1}{r(x)} \E_{x}\Big(r(X_{\tau_n}) e^{\kappa V_{\tau_n}} f \left(X_0,V_0,X_1, V_1, \dots, X_{\tau_n}, V_{\tau_n}\right)\Big),
\label{Def:transformiertesMass:mitJn}
\end{align}
as one can easily see by applying \eqref{Def:transformiertesMass} to $\Ekappa[x]{f(X_0,V_0,X_1, V_1, \dots, X_{k}, V_{k})\1[\{\tau_{n}=k\}]}$ for each $k\ge 1$, which in turn is possible for the appearing indicator is a function of $(X_0,V_0,X_1, V_1, \dots, X_{k}, V_{k})$.

%and thus by (R4)
%\begin{align}
%&\Ekappa[y]{f(X_{\sigma_0},V_{\sigma_0},X_{\sigma_1}, V_{\sigma_1}, \dots, X_{\sigma_n}, V_{\sigma_n})}\nonumber \\ 
%&\hspace{1cm}= \frac{1}{r(y)} \E_{y}\Big(r(X_{\sigma_n}) e^{\kappa V_{\sigma_n}} f \left(X_{\sigma_0},V_{\sigma_0},X_{\sigma_1}, V_{\sigma_1}, \dots, X_{\sigma_n}, V_{\sigma_n} \right)\Big). \label{trafoMass:sigma}
%\end{align}

\begin{lemma}\label{MRT:anwendbar:taun}
The hit chain $(X_{\tau_n})_{n \geq 0}$ constitutes a strongly aperiodic Doeblin chain under $(\Pkappa[x])_{x \in S}$ with stationary distribution $\nu=\pi(\cdot\cap B_{\delta}(x_{0}))/
\pi(B_{\delta}(x_{0}))$. Moreover, $(X_{\tau_n}, V_{\tau_n})_{n \geq 0}$ is a nonarithmetic MRW under $(\Pkappa[x])_{x \in S}$ with positive drift.
\end{lemma}

\begin{proof} 
For the first statement, we just note that $\{\sigma_{n}:n\ge 1\}\subset\{\tau_{n}:n\ge 1\}$.
Next, due to Lemma \ref{lattice-type} and the conditional independence of $U_{1},U_{2},...$ given $(X_{n})_{n\ge 0}$, we find that for $t \neq 0$
\begin{align*}
\Erw[x]{\abs{\Erw{e^{itV_{\tau_1}} | X_0, X_{\tau_1}}}}\ &\le\ \Erw[x]{\prod_{k=1}^{\tau_{1}}|\E\left(e^{itU_{k}}|X_{k-1},X_{k}\right)|}\\
&\le\ \Erw[x]{\abs{\Erw{e^{itV_{1}} | X_0, X_1}}}\ <\ 1
\end{align*}
for $\pi$-almost all and thus $\nu$-almost all $x$. Consequently, $(X_{\tau_n}, V_{\tau_n})_{n \ge 0}$ is nonarithmetic under $(\Prob_x)_{x\in S}$ and $(\Pkappa[x])_{x \in S}$. Finally, we obtain for $\alpha ' := {}^{\kappa}\E_{\nu}V_{\tau_1}$ that
$$ \alpha '\ =\ \lim_{n \to \infty} \frac{V_{\tau_n}}{n}\ \geq\ \lim_{n \to \infty} \frac{V_{\tau_n}}{\sigma_n} \cdot \liminf_{n \to \infty} \frac{\tau_n}{n}\ \geq\ \alpha \cdot 1 \ >\  0 \quad \Pkfs $$
where Remark \ref{slln:harris} should be recalled.
\end{proof}

\begin{proposition}\label{recurrence:times}
Let $x_0 \in S$ and $\delta >0$ be such that $B_\delta(x_0)$ is regenerative with associated regeneration epochs $\sigma_n$, $n\ge 1$.
Then
\begin{equation}\label{lim_stopped}
\liminf_{t \to \infty} t^{\kappa}\,\P[x]{ \sup_{n\ge 1} \abs{x\Pi_{\sigma_n -1}} > t } > 0
\end{equation} 
for $\pi$-almost all $x\in B_\delta(x_0)$.
\end{proposition}

\begin{proof}
Let $(\tau_{n})_{n\ge 1}$ denote the sequence of hitting times of $B_{\delta}(x_{0})$ and observe that it contains $(\sigma_{n}-1)_{n\ge 1}$ as a subsequence. By Proposition \ref{tail:sup:Pitaun}, we have
\[ \lim_{t \to \infty} t^{\kappa}\,\P[x]{\sup_{n \geq 1} \abs{x \Pi_{\tau_n}}>t} = L'\,r(x)>0 ,\]
for some $L' > 0$ and $\pi$-almost all $x\in B_{\delta}(x_{0})$.

\vspace{.1cm}
Fix any such $x$ hereafter and put $\Nhat(t):=\inf\{n\ge 1:\abs{x\Pi_{\tau_{n}}}>t\}$, thus 
$$ \{\sup_{n \geq 1} \abs{x \Pi_{\tau_n}}>t\}=\{\Nhat(t)<\infty\}. $$
Since $\{\sup_{n \in \N} \abs{x\Pi_{\sigma_n -1}} > \epsilon t\}$ contains
$$ A(t):=\bigcup_{n\ge 1}\left\{\Nhat(t)=n,\,\abs{x\Pi_{\tau_{n}+1}}>\epsilon t,\,X_{\tau_{n}+1}\in B_{\delta}(x_{0}),\,J_{\tau_{n}+1}=1\right\} $$
as a subset, it suffices to show that $\liminf_{t\to\infty}t^{\kappa}\,\P[x]{A(t)}>0$ for suitably chosen $\epsilon>0$. To this end, we make the following estimation.
\begin{align*}
\P[x]{A(t)}
&=\ p\sum_{n\ge 1}\P[x]{\Nhat(t)=n,\,\abs{x\Pi_{\tau_{n}+1}}>\epsilon t,\,X_{\tau_{n}+1}\in B_{\delta}(x_{0})}\\
&=\ p\sum_{n\ge 1}\P[x]{\Nhat(t)=n,\,\abs{\esl{x\Pi_{\tau_{n}}}M_{\tau_{n}+1}}>\frac{\epsilon t}{\abs{x\Pi_{\tau_{n}}}},\,X_{\tau_{n}+1}\in B_{\delta}(x_{0})}\\
&\ge\ p\sum_{n\ge 1}\P[x]{\Nhat(t)=n,\,\inf_{y\in S}\abs{yM_{\tau_{n}+1}}>\epsilon,\,X_{\tau_{n}+1}\in B_{\delta}(x_{0})}\\
&\ge\ p\,\inf_{u\in B_{\delta}(x_{0})}\P[u]{\inf_{y\in S}\abs{yM}>\epsilon,\,\esl{uM}\in B_{\delta}(x_{0})}\P[x]{\Nhat(t)<\infty}\\
&\ge\ p\,\inf_{u\in B_{\delta}(x_{0})}\P[u]{\inf_{y\in S}\abs{yM_{\sigma_{1}}}>\epsilon,\,\sigma_{1}=1}\P[x]{\Nhat(t)<\infty}.
\end{align*}
Fixing any $\epsilon\in (0,\mathfrak{c})$, we now infer from \eqref{lower bound} in Lemma \ref{M_sigma:bounded} that
$$ \P[u]{\inf_{y\in S}\abs{yM_{\sigma_{1}}}>\epsilon,\,\sigma_{1}=1}=\P[u]{\sigma_{1}=1}=p>0 $$
for any $u\in B_{\delta}(x_{0})$, whence we finally conclude
$$ \P[x]{A(t)}\ \ge\ p^{2}\,\P[x]{\sup_{n \geq 1} \abs{y \Pi_{\tau_n}}>t} $$
for all $t>0$ and thus $\liminf_{t\to\infty}t^{\kappa}\,\P[x]{A(t)}>0$.
\end{proof}

\section{Proof of Theorem \ref{theorem:main}: Implicit Markov renewal theory} \label{sect:implicitMRT}

We now turn to the proof of our main result Theorem \ref{theorem:main}, all assumptions of which will therefore be in force throughout, in fact in strengthened form given by our standing assumption.

Embarking on ideas by Goldie \cite{Goldie1991} and Le Page \cite{LePage1983}, a comparison of the distribution functions of $xR$ and $xMR$ will enable us to make use of a Markov modulated version of Goldie's implicit renewal theory. This will
prove that $K(x)=\lim_{t\to\infty}\,t^{\kappa}\P{xR>t}$ exists for $\pi$-almost all $x\in S$.
%Hier vielleicht schon Regenerationstechnik f�r $K>0$ erw�hnen?

We start with a simple lemma, stated without proof, which is just Lemma 9.3 in \cite{Goldie1991} adapted to our situation.

\begin{lemma} \label{lem:smooth:convergence}
Let $x\in S$. If $K(x):=\lim_{t \to \infty} t^{-1} \int_{0}^{t} s^{\kappa}\,\P{xR>s} ds$ exists and is finite, then so does $ \lim_{t \to \infty} t^{\kappa}\,\P{xR > t}$ and equals $K(x)$ as well.
\end{lemma}

%\begin{proof}
%Fix $b >0$. Then
%\begin{align*}
%& e^{-t} \frac{b e^{-(\kappa+1)b}}{\kappa +1} e^{(\kappa +1) t} \P{xR > e^t} \leq e^{-t} \int_{t-b}^t e^{(\kappa+1)v} \P{xR>e^v} dv \\
%\leq & e^{-t} \left(1- e^{-(\kappa+1)b} \right) \int_{-\infty}^t e^{(\kappa+1)v} \P{xR>e^v} dv \stackrel{t \to \infty}{\longrightarrow} \left(1- e^{-(\kappa+1)b} \right) K(x),
%\end{align*}
%so $ \limsup_{t \to \infty} e^{\kappa t} \P{xR>e^t} \leq K(x) (\kappa+1) \frac{1- e^{-(\kappa+1)b}}{b e^{-(\kappa+1)b}}$ . Letting $b \downarrow 1$ one sees that the $\limsup$ is at most $K(x)$. The $\liminf$ part is similar and omitted here.
%\end{proof}

%\begin{lemma} \label{lem:smooth:convergence} \ \\
%If $\lim_{t \to \infty} t^{-1} \cdot \int_{0}^{t} s^{\kappa} \P{xR>s} ds  = K(x)$, then $ \lim_{t \to \infty} t^\kappa \P{xR > t} = K(x)$.
%\end{lemma}
%
%\begin{proof}
%Fix $0 < b < 1$. Then
%\begin{equation*} t^{-1} \frac{1-b^{\kappa+1}}{\kappa+1} t^{\kappa+1} \P{xR>t} \leq t^{-1} \int_{bt}^t s^\kappa \P{xR > s} ds \to K(x) (1-b) , \ \ t \to \infty,
%\end{equation*}
%so $ \limsup_{t \to \infty} t^\kappa \P{xR>t} \leq K(x) (\kappa+1) \frac{1-b}{1-b^{\kappa+1}}$ .
%Letting $b \uparrow 1$ one sees that the $\limsup$ is at most $K(x)$. The $\liminf$ part is similar and can be found in  \cite[lemma 9.3]{Goldie1991}.
%\end{proof}

Substituting $t'$ for $e^{t}$ and a change of variables show that $t'^{-1}\int_{0}^{t'}s^{\kappa}\,\P{xR>s}\,ds$ equals $e^{-t}\int_{-\infty}^{t}e^{(\kappa+1)s}\,\P{xR>e^{s}}\,ds$ which is the form needed in the next result which provides us with the basic renewal theoretic identity.

\begin{lemma} \label{lem:erneuerungsbeziehung}
For all $t \in \R$,
\begin{align}\label{erneuerungsbez:exp}
\frac{e^{-t}}{r({x})} \int_{-\infty}^t \!\!e^{(\kappa+1)s}\,\P{xR>e^s}\, ds = & 
\sum_{n\ge 0} \int  \widehat{g}(y,t-u)\,\Pk[x]{X_n \in dy,V_n \in du} ,
\end{align}
where  $\widehat{g}(y,t) :=  \int_{-\infty}^t e^{-(t-s)} g(y,s)\,ds$ is the exponential smoothing of

\[ g(y,s)= \frac{e^{\kappa s}}{r({y})} \left[ \P{yR > e^s} - \P{yMR > e^s} \right]  .\]
\end{lemma}

\begin{proof}
%By substituting $s=e^v$ and $t'=e^t$, the LHS is equal to
%\[ e^{-t} \int_{-\infty}^t \frac{e^{(\kappa+1)v} \P{xR>e^v}}{r({x})} dv = \int_{-\infty}^t e^{-(t-v)} \frac{1}{r({x})} e^{\kappa v} \P{xR>e^v}dv . \]
%We start by considering the innermost part of the left hand side, $\P{xR > e^v}$, and will later add the rest and convolute with $Exp(1)$:
For arbitrary $n \in \N$, $x \in S$ and $s \in \R$, consider the following telescoping sum  for $\P{xR> e^s}$ (recalling independence of $R$, $M$, and $(M_n)_{n \ge 1}$)
\begin{align*}
\sum_{k=1}^n &\left[ \P{x \Pi_{k-1} R> e^s} - \P{x \Pi_k R> e^s} \right] + \P{x \Pi_n R > e^s} \\
		 = & \sum_{k=1}^n \left[ \P[x]{e^{V_{k-1}}X_{k-1} R > e^s} - \P[x]{e^{V_{k-1}}X_{k-1}M_k R > e^s} \right] + \P[x]{e^{V_{n}}X_{n} R > e^s} \\
		 = & \sum_{k=0}^{n-1} \left[ \P[x]{e^{V_{k}}X_{k} R > e^s} - \P[x]{e^{V_{k}}X_{k}M R > e^s} \right] + \P[x]{e^{V_{n}}X_{n} R > e^s}\\
		 = & \sum_{k=0}^{n-1} \int \P{yR > e^{s-u}} - \P{yMR > e^{s-u}}\,\Prob_x(X_k\in dy, V_k\in du) \\
		 &\hspace{8cm}+ \P[x]{e^{V_{n}}X_{n} R > e^s}
\end{align*}
Multiply by $e^{\kappa s}/r({x})>0$ to obtain
\begin{align*}
\frac{e^{\kappa s}}{r({x})}\,&\P{xR> e^s}\\
&=\  \sum_{k=0}^{n-1} \int \frac{e^{\kappa(s-u)}}{r({x})} \left[ \P{yR > e^{s-u}} - \P{yMR > e^{s-u}} \right] \\   
&\hspace{1.5cm}\times\frac{r({y})}{r({y})}\,e^{\kappa u}\,\Prob_x(X_k\in dy, V_k\in du)  + \frac{e^{\kappa s}}{r({x})}\,\P[x]{e^{V_{n}}X_{n} R > e^s} \\
& =\ \sum_{k=0}^{n-1} \int g(y,s-u) \Probk_x (X_k \in dy , V_k \in du)  + \frac{e^{\kappa s}}{r({x})}\,\P[x]{X_{n} R > e^{s-V_n}}.
\end{align*}
Convolution with a standard exponential distribution then gives
\begin{align*}
\int_{-\infty}^t e^{-(t-s)} \frac{1}{r({x})} e^{\kappa s}\,\P{xR>e^s}\,ds\ 
&=\   \sum_{k=0}^{n-1}\ \int \widehat{g}(y,t-u)\,\Pk[x]{X_k \in dy,V_k \in du} \\
&+\ \int_{-\infty}^t e^{-(t-s)}\,\frac{e^{\kappa s}}{r({x})}\,\P[x]{X_{n} R > e^{s-V_n}} ds
\end{align*}
By the Cauchy-Schwarz inequality, $\abs{X_n R}\le |X_{n}|\,|R|=|R|$ and thus
\[ \P[x]{X_{n} R > e^{s-V_n}}\ \leq\ \P[x]{\abs{X_n R} > e^{s-V_n}}\ \leq\ \P[x]{\abs{R} > e^{s -V_n}}. \]
But the last term converges to 0 as $n\to\infty$ for any $s>0$, because $\lim_{n \to \infty} V_n = - \infty$ $\Prob_x$-a.s. Hence assertion \eqref{erneuerungsbez:exp} follows by an appeal to the dominated convergence theorem.
\end{proof}

Obviously, if ${}^{\kappa}\Bbb{U}_{x}:=\sum_{n\ge 0}\Pkappa[x]((X_{n},V_{n})\in\cdot)$, then the right-hand side of \eqref{erneuerungsbez:exp} equals $\widehat{g}*{}^{\kappa}\Bbb{U}_{x}(t)$ for $x$ outside a $\pi$-null set $N$ provided that sum and integral may be interchanged for $x\not\in N$. But the latter follows if we can prove hereafter that $\widehat{g}$ is $\pi$-directly Riemann integrable which will also be the crucial condition that ensures applicability of the MRT \ref{MRT}.
Indeed, if $\widehat{g}$ has this property, then, by Equation (5.8) and Lemma A.5 in \cite{Als1997},
\begin{align*}
\widehat{g}*{}^\kappa\Bbb{U}_x(t)\ &=\ \Ek[x]{\sum_{k\ge 0} \widehat{g}(X_k, t- V_k)}\\
&= \int\sum_{k\ge 0} \widehat{g}(y,t-u)  \Pk[x]{X_k \in dy,V_k \in du}\ <\ \infty
\end{align*}
for all $t \in \R$ and $\pi$-almost all $x\in S$. Split $\widehat{g}$ in positive and negative part. This yields two $\pi$-null sets $N_1, N_2$ such that $\widehat{g}^+*{}^\kappa\Bbb{U}_x(t)$ and $\widehat{g}^-*{}^\kappa\Bbb{U}_x(t)$ are finite for all $x \in (N_1 \cup N_2)^c$ and all $t \in \R$. By Fubini's theorem, sum and integral in \eqref{erneuerungsbez:exp} may be interchanged for all $x \in (N_1 \cup N_2)^c$. This is enough because the MRT asserts convergence of $\widehat{g}*{}^\kappa\Bbb{U}_x(t)$ only for $x$ outside a $\pi$-null set.

Instead of $\pi$-direct Riemann integrability of $\widehat{g}$ we will actually show the stronger property that 
\begin{equation} \label{wudRi} 
\sup_{y \in S}\ \sum_{n \in \Z}\ \sup_{t \in [n \delta, (n+1) \delta)} \abs{\widehat {g}(y,t)} < \infty ,\end{equation}
which can be done by resorting to the methods of Goldie \cite[proof of Theorem 4.1]{Goldie1991} which are only summarized here. Let $L_{1}(\R)$ as usual be the space of Lebesgue integrable functions.

\begin{lemma} \label{lem:dRi1} %cite[lemmas 9.1 \& 9.2]{Goldie1991}
If $f \in L_1(\R)$ and $\widehat{f}(t):=\int_{- \infty}^t e^{-(t-u)} f(u)\, du$, then for any $\delta >0$
\[ \sum_{n \in \Z} \sup_{t \in [n\delta, (n+1) \delta)} \abs{\widehat{f}(t)}\ \leq\ \delta e^{2 \delta} \int \abs{f(t)}\, dt < \infty. \]
\end{lemma}

\begin{proof} This is Lemma 9.2 in \cite{Goldie1991}
%W.l.o.g.\ let $f$ be nonnegative. For all $\epsilon >0$ and $t \in \R$, we have
%\[ \widehat{f}(t+\epsilon)  \geq e^{-t - \epsilon} \int_{-\infty}^t e^u f(u)\, du = e^{-\epsilon} \widehat{f}(t) .\]
%Consequently, with $m_{n,\delta}:=\sup_{t \in [n \delta, (n+1)\delta)} \widehat{f}(t)$,
%\[ \delta \sum_{n \in \Z} m_{n,\delta} \leq \delta e^{\delta} \sum_{n \in \Z} \widehat{f}(n \delta) \leq \delta e^{2 \delta} \sum_{n \in \Z} \int_{(n-1)\delta}^{n \delta} f(t)\,dt = \delta e^{2 \delta} \int f(t)\,dt<  \infty \]
%for any $\delta >0$.
\end{proof}

In view of the previous lemma, it suffices to show for \eqref{wudRi} that $\int\abs{g(y,s)}\,ds$ is uniformly bounded in $y$. First observe that (cf.\ \cite[Corollary 2.4]{Goldie1991})
\begin{align*}
\int_{\R} \abs{g(y,s)}\,ds= & \int_{\R} \frac{e^{\kappa s}}{r({y})} \abs{ \P{yR > e^s} - \P{yMR > e^s}}\,ds  \nonumber \\
= & \int_{\R} \frac{e^{\kappa s}}{r({y})} \abs{ \P{yMR + yQ > e^s} - \P{yMR > e^s}} \,ds  \nonumber \\
= &  \frac{1}{\kappa r(y)} \E{\abs{((yMR+yQ)^+)^\kappa - ((yMR)^+)^\kappa}}.
 \end{align*}
Then a case-by-case analysis with respect to the signs of $yMR$ and $yQ$ yields that
\[ \sup_{y \in S} \frac{1}{\kappa r(y)} \E \abs{((yMR+yQ)^+)^\kappa - ((yMR)^+)^\kappa} < \infty, \]
see \cite[Theorem 4.1]{Goldie1991}. %resp.\ Lemma \ref{lem:pi:dRi}.
Now we  are ready to prove

\begin{lemma} \label{Konvergenzlemma}
For $\pi$-almost all $x \in S$,
\begin{align*} 
\lim_{t \to \infty}\frac{e^{-t}}{r({x})} \int_{-\infty}^t e^{(\kappa+1)s}\,\P{xR>e^s} ds=  K_{0},
\end{align*}
where $K_{0}:=\frac{1}{\alpha \kappa} \int_{S} \frac{1}{r({y})}  \Erw{((yR)^+)^{\kappa}-((yMR)^+)^{\kappa}} \pi(dy)<\infty$ and $\alpha$ as before denotes the drift of $(X_{n},V_{n})_{n\ge 0}$.
\end{lemma}

\begin{proof}
Since $\widehat{g}$ is $\pi$-directly Riemann integrable, we may exchange sum and integral in \eqref{erneuerungsbez:exp} for $\pi$-almost all $x \in S$ and apply the MRT. This tells us that the right-hand side of \eqref{erneuerungsbez:exp} has the finite limit 
\begin{align*}
& \frac{1}{\alpha} \int_{S} \int_{\R} e^{-t} \int_{- \infty}^t \frac{e^{(\kappa+1)s}}{r({y})} \left[ \P{yR > e^s} - \P{yMR>e^s} \right] ds\ dt \,\pi(dy) \\
&=  \frac{1}{\alpha} \int_{S} \int_{\R} g(y,t)\ dt\ \pi(dy)\\
&=  \frac{1}{\alpha} \int_{S} \frac{1}{r({y})} \int_{\R}  e^{\kappa t} \left[ \P{yR > e^t} - \P{yMR>e^t} \right]\ dt\ \pi(dy) \\
&=   \frac{1}{\alpha} \int_{S} \frac{1}{r({y})} \int_0^\infty  u^{\kappa-1} \left[ \P{yR > u} - \P{yMR>u} \right]\ du \ \pi(dy) \\
&=  \frac{1}{\alpha \kappa} \int_{S} \frac{1}{r({y})}  \Erw{((yR)^+)^{\kappa}-((yMR)^+)^{\kappa}} \pi(dy)
\end{align*}
for $\pi$-almost all $x$.
\end{proof}

%\begin{remark} \label{allex}
%Assumption \eqref{A4} implies that $\pi(U)>0$ for every open set $U  \subset S$. So $\supp(\pi)=S$. Since for every $x \in S$, $\P{xR>t}$ is a continuous function of $t$ by lemma \ref{lemma:subspace}, and with $x_n \to x$ also $x_n R \to R$ $\Pfs$, we have as in the Glivenko-Cantelli-theorem,
%\[ \lim_{n \to \infty} \sup_{t \in \R} \abs{\P{x_n R > t} - \P{xR > t}}=0 .\] 
%Be $\epsilon >0$ and $x \in N$, $N$ the set where the convergence is not implied by the MRT, and choose a sequence $(x_n)$ with $x_n \notin N$ converging to $x$ and such that  $\sup_t \abs{\P{x_n R > t} - \P{x R > t}} < 2^{-\kappa n}$. By continuity of $r$, there is at least a subsequence such that $\abs{r(x_n)-r(x)} < \epsilon$.we can at last infer that lemma \ref{Konvergenzlemma} holds for all $x \in S$. 
%\end{remark}

% Problem: brauche lokal gleichm��ige Konvergenz von t^\kappa P(xR>t) gegen K(x)!! die gibts i.A. nicht.
% vielleicht mit Fuh+Lai.

\section{Proof of Theorem \ref{theorem:main}: Assertion \eqref{theorem:main:assertion2} holds for all $x\in S$}
\label{sect:remove null set}

So far we have proved our main assertion \eqref{theorem:main:assertion2} (except for the positivity of $K(x)$) for $\pi$-almost all $x\in S$ and thus for all $x$ from a dense subset of $S$ (this is a direct consequence of \eqref{A4}). By employing a refined renewal argument, we will now remove this restriction. To this end, we fix an arbitrary $x\in S$ and $\delta>0$ so small that $B_{\delta}(x)$ is regenerative for $\Pkern[\kappa]$ with minorizing distribution $\phi$  and associated sequence $(\sigma_{n})_{n\ge 1}$ of regeneration epochs such that Lemma \ref{M_sigma:bounded} is in force. Put $\sigma:=\sigma_{1}$. The task is to show that $\widehat{g}*{}^\kappa\Bbb{U}_x(t)$ converges to $K_{0}$, and we begin by pointing out that 
\begin{equation}\label{renewal decomposition}
\widehat{g}*{}^\kappa\Bbb{U}_x(t) = \Ek[x]{\sum_{k\ge 0} \widehat{g}(X_k, t- V_k)}
=G(x,t)+\widehat{g}*{}^\kappa\Bbb{U}_{\varphi(x, \cdot)}(t)
\end{equation}
where $\varphi(x,\cdot):=\Pkappa[x]{((X_{\sigma},V_{\sigma})\in\cdot)}$ and
$$ G(x,t) := \Ek[x]{\sum_{k=0}^{\sigma-1}\widehat{g}(X_k, t- V_k)}, \quad (x,t)\in S\times\R. $$
As for this last function, we now prove:

\begin{lemma} 
The function $G$ is bounded and satisfies $\lim_{t\to\infty}G(y,t)=0$ for all $y\in S$.
\end{lemma}

\begin{proof}
By \eqref{wudRi}, $C:=\sup \{ \abs{\widehat{g}(y,t)}:y \in S,\, t \in \R \}< \infty$, and since $(X_{n})_{n\ge 0}$ is a strongly aperiodic Doeblin chain, we infer
$$ \sup_{y\in S,t\in\R}|G(y,t)| \le C\,\sup_{y\in S}\Ekappa[y]\sigma <\infty. $$
Just note that the time it takes to hit the regenerative ball $B_{\delta}(x)$ pertaining to $\sigma$ from any $y$ is geometrically bounded (uniformly in $y$) and that a geometric number of coin tosses (the $J_n$) of such times determines $\sigma$. Turning to the convergence assertion, we point out that, again by property \eqref{wudRi}, $\lim_{t\to\infty}\widehat{g}(y,t)=0$ for all $y \in S$, which implies the desired result by an appeal to the dominated convergence theorem.
\end{proof}

In view of \eqref{renewal decomposition}, we are now left with a proof of $\widehat{g}*{}^\kappa\Bbb{U}_{\varphi(x,\cdot)}(t)\to K_{0}$ defined in Lemma \ref{Konvergenzlemma}. This requires one more lemma.

\begin{lemma}\label{independence}
The sequence $(X_{\sigma},(X_n, U_{n})_{n > \sigma})$ is independent of $(X_{\sigma-1},V_{\sigma-1})$ under $\Pkappa[x]$ with distribution given by $\Pkappa[\phi]((X_0,(X_{n},U_{n})_{n \ge 1})\in\cdot)$.
\end{lemma}

\begin{proof}
The first assertion follows directly when observing that, by regeneration, $(X_{\sigma+n})_{n \ge 0}$ and $(X_{\sigma-1},V_{\sigma-1})$ are independent under $\Pkappa[x]$, and the fact that the con\-ditional distribution of $U_{k}$ given $(X_n)_{n \geq 0}$ only depends on $(X_k, X_{k-1})$. The proof is completed by the observation that $\Pkappa[x]((X_{\sigma+n})_{n \ge 0} \in\cdot)=\Pkappa[\phi]((X_{n})_{n \ge 0})\in\cdot) $.
\end{proof}

Define $V_{\sigma,n}:=V_{\sigma+n}-V_{\sigma}$ for $n\ge 0$ and then
$$ h(x,s,t) := \Ekappa[x]\Bigg(\sum_{k\ge 0}\widehat{g}(X_{\sigma+k},t-s-V_{\sigma-1}-V_{\sigma,k}) \Bigg) $$
for $s,t\in\R$. Lemma \ref{independence} implies
$$ h(x,s,t) = \int_{\R} \widehat{g}*{}^\kappa\Bbb{U}_{\phi}(t-s-r)\,\Pkappa[x](V_{\sigma-1}\in dr). $$
As $\widehat{g}$ satisfies \eqref{wudRi}, we infer from the MRT \ref{MRT} and the subsequent remark that $\widehat{g}*{}^\kappa\Bbb{U}_{\phi}(t)$ is bounded and converges to $K_{0}$. By the dominated convergence theorem, the same limit holds for $\lim_{t\to\infty}h(x,s,t)$ for all $s$.

Finally, the connection between $h(x,s,t)$ and $\widehat{g}*{}^\kappa\Bbb{U}_{\varphi(x, \cdot)}(t)$ becomes apparent after the following observations: By Lemma \ref{M_sigma:bounded},
$U_{\sigma}$ is taking values in some finite interval $[s_{*},s^{*}]$. Hence we can estimate $\widehat{g}*{}^\kappa\Bbb{U}_{\varphi(x, \cdot)}(t)$ by
\begin{equation*}
\inf_{s\in [s_{*},s^{*}]}h(x,s,t) \le \widehat{g}*{}^\kappa\Bbb{U}_{\varphi(x, \cdot)}(t) \le \sup_{s\in [s_{*},s^{*}]}h(x,s,t)
\end{equation*}
and thus arrive at the desired conclusion that $\lim_{t\to\infty}\widehat{g}*{}^\kappa\Bbb{U}_{\varphi(x, \cdot)}(t)=K_{0}$.

\section{Proof of Theorem \ref{theorem:main}: The limit $K(x)$ is positive}\label{sect:positive}

A combination of Lemma \ref{lem:smooth:convergence}, Lemma \ref{Konvergenzlemma} and the result of Section \ref{sect:remove null set} renders convergence of $t^{\kappa}\Prob(xR>t)$ to the continuous function
$$ K(x):= K_{0}\,r(x) = \frac{r(x)}{\alpha \kappa} \int_{S} \frac{1}{r({y})}  \Erw{((yR)^+)^{\kappa}-((yMR)^+)^{\kappa}} \pi(dy) $$
for all $x\in S$. To complete the proof of Theorem \ref{theorem:main}, it remains to show that $K$ or, equivalently, $K_{0}$ is positive, which is the topic of this final section. 

Clearly, it suffices to show that $ \limsup_{t \to \infty} t^\kappa \P{xR>t} > 0 $ for some $x\in S$, or equivalently (since the limit exists) that the $\liminf$  is positive. Notice that, as $r$ is symmetric (Lemma \ref{lem:r:symmetric}), the same holds true for $K(x)$, hence
\[ \lim_{t \to \infty} \P{xR>t}= \lim_{t \to \infty} \P{-xR>t} = \frac{1}{2} \lim_{t \to \infty} \P{\abs{xR} >t}. \]
So it is enough to show that $\liminf_{t \to \infty} t^\kappa \P{\abs{xR}>t} >0$ for some $x$. To this end we need the following lemma, originally due to Le Page \cite[Lemma 3.11]{LePage1983}, which ensures that $R$ and its ''marginals'' $xR$ for any $x\in S$ have unbounded support. It is this result where the nondegeneracy assumption \eqref{A6}, unused so far, enters in a crucial way. We postpone the proof until the end of this section. 

\begin{lemma}  \label{lem:R:unbounded}
For all $x \in S$ and $t \in \R$,
\begin{equation} \label{R:unbounded}
\P{xR \leq t } < 1 .
\end{equation}
\end{lemma}

What this lemma shows is that, fixing any $x_{0}\in S$, we can choose $\xi>0$ and then sufficiently small $\zeta,\eta\in (0,1)$ and $\delta\in (0,\zeta)$ such that
\begin{align}  
\P{zR > \xi} \geq \eta\quad\text{and}\quad
 \P{zR < (1-\zeta) \xi} \geq \eta \label{Wahl:xi} 
 \end{align}
for all $z \in B_\delta(x_0)$. Notice that
\begin{equation}\label{pos scalar product}
\inf_{z,y\in B_\delta(x_0)}zy> 1- \delta .
\end{equation}

We continue with a decomposition of $xR$ with respect to entrances of $\esl{x\Pi_k}$ into $B_\delta(x_0)$. In the following lemma, consider any (sub-)sequence $(\sigma_n)_{n \geq 1}$ of the hitting times (e.g. regeneration times). Note that \eqref{Wahl:xi} particularly holds for $z= X_{\sigma_n}= \esl{x \Pi_{\sigma_n}}$. Recall from Subsection \ref{subsect:stopped eq} the definition of $Q^{n}$ and $R^{n}$ as well as $R^{\tau}\stackrel{d}{=}R$ for any a.s.\ finite stopping time $\tau$ with respect to $(\mathcal{F}_{n})_{n\ge 0}$, the natural filtration of $(M_{n},Q_{n})_{n\ge 1}$.

\begin{lemma} \label{lem:Levy:ineq}
Given any $x_0 \in S$ and sufficiently small $0 <\delta <1$,
\[ \P{\abs{xR} > t}\ \geq\ \eta \, \P[x]{\sup_{n \geq 1} \abs{ x Q^{\sigma_n} + \xi\,x \Pi_{\sigma_n} y} >t } \]
holds true for all $x\in S$ and $y\in B_{\delta}(x_{0})$.
\end{lemma}

\begin{proof}
This is an extension of Levy's inequality and inspired by \cite[Prop. 4.2]{Goldie1991}. Since $R^{\sigma_{k}}\stackrel{d}{=}R$ for all $k\ge 1$ we see that \eqref{Wahl:xi} holds for $R^{\sigma_k}$ as well. We show first that
\[ \P{xR > t}\ \geq\ \eta\, \P[x]{\sup_{n \geq 1}  x Q^{\sigma_n} + \xi\,x \Pi_{\sigma_n} y >t } \]
and will consider $\P{-xR>t}$ in a second step. Define
\begin{align*}
C_k := & \left\{ \max_{1 \leq j < k} \left( xQ^{\sigma_j} + \xi\, x \Pi_{\sigma_j}y\right) \leq t,\,xQ^{\sigma_k} + \xi\,x \Pi_{\sigma_k}y > t	\right\}\\
\text{and}\quad
D_k := & \left\{	x \Pi_{\sigma_k}R^{\sigma_k} > \xi\,x \Pi_{\sigma_k}y\right\} . % = \left\{(x \Pi_{\sigma_k})^{\sim}R^{\sigma_k} > \xi\,(x \Pi_{\sigma_k})^{\sim}y\right\}.
\end{align*}
By \eqref{pos scalar product}, $0< \esl{x \Pi_{\sigma_k}}y \le 1$ for all $y \in B_ \delta(x_0)$, giving 
$$ D_{k}  = \left\{ \esl{x\Pi_{\sigma_k}}R^{\sigma_k} > \xi \esl{x \Pi_{\sigma_k}}y\right\} \supset \left\{\esl{x \Pi_{\sigma_k}} R^{\sigma_k} > \xi\right\}  $$
 and thus $\P[x]{D_k | \mathcal{F}_{\sigma_k}} \geq \eta$ $\Prob_{x}$-a.s. In combination with $\bigcup_{k=1}^n (C_k \cap D_k) \subset \{ xR > t \}$ 
and $C_{k}\in\mathcal{F}_{\sigma_k}$, this implies
\begin{equation*}
\Prob(xR > t)\ \geq\ \sum_{k=1}^n \int_{C_k} \P[x]{D_k | \mathcal{F}_{\sigma_k}} d\Prob_{x} \ \geq\ \eta \,\P[x]{\bigcup_{k=1}^n C_k} ,
\end{equation*}
and thus $\Prob(xR > t) \geq \eta\, \P[x]{\sup_{n \geq 1}\left(x Q^{\sigma_n} + \xi x \Pi_{\sigma_n} y\right)  >t } $
by letting $n\to\infty$. 

Turning to the respective inequality for $\P{-xR>t}$, define
\begin{align*}
C_k' := & \left\{ \max_{1 \leq j < k} \left(- xQ^{\sigma_j} - \xi\, x \Pi_{\sigma_j}y\right) \leq t,\, -xQ^{\sigma_k} - \xi\,x \Pi_{\sigma_k}y > t	\right\}\\
\text{and}\quad
D_k' := & \left\{	- x \Pi_{\sigma_k}R^{\sigma_k} > - \xi\,x \Pi_{\sigma_k}y\right\} = \left\{ \esl{x\Pi_{\sigma_k}}R^{\sigma_k} < \xi \esl{x \Pi_{\sigma_k}}y\right\}. % = \left\{(x \Pi_{\sigma_k})^{\sim}R^{\sigma_k} > \xi\,(x \Pi_{\sigma_k})^{\sim}y\right\}.
\end{align*}
Again by \eqref{pos scalar product}, $\esl{x \Pi_{\sigma_k}}y\ge 1-\delta>1-\zeta $ for all $y \in B_ \delta(x_0)$, giving 
$$ D_{k}'  \supset \left\{\esl{x \Pi_{\sigma_k}} R^{\sigma_k} < (1- \zeta) \xi\right\}  $$
and thus $\P[x]{D_k ' | \mathcal{F}_{\sigma_k}} \geq \eta$ $\Prob_{x}$-a.s. Now reasoning as above, 
\begin{align*}
\Prob(-xR > t)\ \geq\ \eta \,\lim_{n\to\infty}\,\P[x]{\bigcup_{k=1}^n C_k'}
\ =\ \eta\, \P[x]{\sup_{n \geq 1}\left(-x Q^{\sigma_n} - \xi x \Pi_{\sigma_n} y\right)  >t }
\end{align*}
The desired result hence follows by a combination of this inequality with the one obtained for $\P{xR>t}$.
 \end{proof}

\begin{proposition}\label{prop:positive limit}
There exists $x\in S$ such that $\liminf_{t\to\infty}t^{\kappa}\,\Prob(|xR|>t)$ is positive. 
\end{proposition}

\begin{proof}
Pick any regenerative $B_{\delta}(x_{0})$ with $\delta$ sufficiently small, such that Lemma \ref{lem:Levy:ineq} holds true, and let $\sigma_1,\sigma_{2},...$  be the associated regeneration times, thus $X_{\sigma_{n}}\stackrel{d}{=}\phi$ for $n\ge 1$.
By Proposition \ref{recurrence:times}, $\liminf_{t \to \infty} t^\kappa \P[x]{\sup_{n \geq 1} \abs{x \Pi_{\sigma_{n}-1}}> t} $ is positive for $\pi$-almost all $x\in B_{\delta}(x_{0})$.
Fix any such $x$ hereafter.

Define $\Pi_{j,k}:=M_{j}\cdot...\cdot M_{k}$, $Q^{j,n} := \sum_{k=j}^n \Pi_{j,k-1} Q_k$ and (with $\sigma_{0}:=0$ and any $y\in B_{\delta}(x_0)$, to be chosen in Lemma \ref{lemma:nondegenerate:stoppingtime})
\begin{align*}
& T_n = xQ^{\sigma_n } + \xi\,x\Pi_{\sigma_n }y,\\
& \Delta_{n}:=Q^{\sigma_{n-1}+1 , \sigma_{n}} - \xi\,(I- \Pi_{\sigma_{n-1}+1, \sigma_{n}})\,y, \\
& U_n = x \Pi_{\sigma_{n-1}}\Delta_{n}
\end{align*}
for $n\ge 1$.
Then $T_n = T_{n-1} + U_n$ and $\{\sup_{n \geq 1} \abs{T_n} > t \} \supset \{\sup_{n \geq 2} \abs{U_n} > 2t \}$. Lemma \ref{lem:Levy:ineq} provides us with
\[ \P{\abs{xR} > t} \geq \eta\, \P[x]{\sup_{n \geq 1} \abs{T_n} >t } \]
for some $\eta>0$.

By Lemma \ref{M_sigma:bounded}, $\inf_{y \in S} \abs{yM_{\sigma_{n}}}\ge\mathfrak{c}$ a.s.\ for all $n\ge 1$ and a suitable $\mathfrak{c}>0$.
Hence, for all $t >0$,  %cp. Goldie, p. 157
\begin{align}
&\P{\abs{xR} > t} \geq \eta\,\P[x]{ \sup_{n \geq 2} \abs{U_n} \geq 2t }
\nonumber\\
&\quad= \eta\,\P[x]{ \sup_{n \geq 1} \abs{x\Pi_{\sigma_{n}-1}}\abs{(x\Pi_{\sigma_{n}-1 })^{\sim}M_{\sigma_{n}}}\abs{X_{\sigma_{n}}\Delta_{n+1}} \geq 2t }\nonumber \\
&\quad\geq  \eta\, \sum_{n \geq 1} \P[x]{\bigcap_{k=1}^{n-1} A_k , \ \abs{x\Pi_{\sigma_{n}-1 }} >  \frac{2t}{\mathfrak{c}\epsilon}, \abs{X_{\sigma_n}\Delta_{n+1}} > \epsilon}\nonumber \\
&\quad\geq \eta\, \sum_{n \geq 1} \P[x]{ \bigcap_{k=1}^{n-1} A_k, \ \abs{x\Pi_{\sigma_{n}-1 }} > \frac{2t}{\mathfrak{c}\epsilon}}  \P[\phi]{\abs{X_0\Delta_{1}} > \epsilon}\tag{\text{use (R3)}}\\
&\quad\geq  \eta\, \P[\phi]{\abs{ X_0\Delta_{1}} > \epsilon}\, \P[x]{\sup_{n \geq 1} \abs{x \Pi_{\sigma_n -1}} > \frac{2t}{\mathfrak{c}\epsilon}},\nonumber
\end{align} 
where $ A_k = \{\abs{x\Pi_{\sigma_{k}-1}} \leq2t/(\mathfrak{c}\epsilon) \} $
for $k\ge 1$ and some fixed $0<\epsilon<1$. The proof is finished by the subsequent lemma will where we show positivity of $\P[\phi]{\abs{ X_0\Delta_{1}} > \epsilon}$. Together with \eqref{lim_stopped} this clearly yields the desired conclusion.
\end{proof}

\begin{lemma} \label{lemma:nondegenerate:stoppingtime}
In the situation of Proposition \ref{prop:positive limit}, there exist $\epsilon >0$ and $y\in B_{\delta}(x_{0})$ such that (notice here the dependence of $\Delta_{1}$ on $y$)
\[  \P[\phi]{\abs{ X_0\Delta_{1}} > \epsilon} > 0 .\]
\end{lemma}

\begin{proof}
Suppose that $X_{0}\Delta_{1}=X_{0}(Q^{\sigma}-\xi\,(I-\Pi_{\sigma})\,y)=0$ $\Prob_{\phi}$-a.s.\ for all $y\in B_{\delta}(x_{0})$, where $\sigma:=\sigma_{1}$. 
Then the same holds true for all $y$ in the convex hull of $B_{\delta}(x_{0})$ (as a subset of $\R^{d}$) which contains a basis of $\R^d$. Consequently, the range of $Q^{\sigma}-\xi\,(I-\Pi_{\sigma})$ and $\{tX_{0}:t\in\R\}$ are orthogonal $\Prob_{\phi}$-a.s. On the other hand, by Lemma \ref{contractive},
$\Prob_{\phi}(\|\Pi_{\sigma}\|<1) > 0$ and thus $Q^{\sigma}-\xi\,(I-\Pi_{\sigma})$ has full range $\R^{d}$ on a set of positive probability under $\Prob_{\phi}$. This contradicts our starting assumption and the lemma is proved.
\end{proof}

We close this section with a proof of Lemma \ref{lem:R:unbounded}.

\begin{proof}[Proof of Lemma \ref{lem:R:unbounded}$\,$]
We first show that $\supp R$ is not a compact subset of $\R^d$. Use \eqref{eq:R:iterate} to infer for each $n\ge 1$,
\[ \Pi_n \supp{ R } + Q^n = \supp R\quad\Pfs \]
and thus also $\Pkfs$, for $\Prob_{x}$ and $\Pkappa[x]$ for any $x$ are equivalent probability measures on each $\mathcal{F}_{n}=\sigma((M_{j},Q_{j})_{1\le j\le n})$, $n\ge 1$. Now assume, that $\supp R$ is bounded. By \eqref{A6}, there exist at least two distinct $x_1, x_2 \in \supp{R}$. Defining $v:=x_1 - x_2$, it then follows that for all $n \geq 1$ and some $C\in (0,\infty)$
\[ \abs{\Pi_n v } \leq \abs{\Pi_n x_1 + Q^n} + \abs{\Pi_n x_2 +Q^n} < C\quad\Pkfs \]
and thereupon for all $x\in S$
$$ C \ge \abs{x\Pi_{n}v} = \abs{x\Pi_{n}}\abs{(x\Pi_{n})^{\sim}v}\quad\Pkfs $$
The hitting times $\tau_n$ of $\esl{x \Pi_n}$ in $B_\delta(v)$ are $\Pkfs[x]$-finite, yielding 
$$ \limsup_{n \to \infty} \abs{x \Pi_{\tau_n}} \leq \frac{C}{X_{\tau_n} v} \leq \frac{C}{1-\delta} \quad\Pkfs[x] $$
for all $x\in S$, where \eqref{pos scalar product} should be recalled for the final bound. Consequently,
$$ \limsup_{n \to \infty} \frac{V_{\tau_n}}{\tau_n} = \limsup_{n \to \infty} \frac{1}{\tau_n} \log \abs{X_0 \Pi_{\tau_n}} = 0 \quad \Pkfs[\pi] $$
which contradicts Lemma \ref{MRT:anwendbar:taun}.

Having thus shown that $\supp R$ is not compact in $\R^d$, there exist sequences $(x_n)_{n \geq 1} \subset \supp R$ with $\lim_{n \to \infty} \abs{x_n} = \infty$ whence, by compactness of $S$, the following set is nonempty:
\[ D := \left\{ y \in S \ : \ \exists \, (x_n)_{n \geq 1} \subset \supp R, \ \lim_{n \to \infty} \abs{x_n}= \infty, \ \lim_{n \to \infty} x_n^{\sim} =y \right\}. \]
Now suppose that $\P{x_0 R \leq t_0}=1$ for some $(x_0, t_0) \in S \times \R$. For any $y_0 \in D$, choose a sequence $(x_n)_{n\ge 1} \subset \supp{R}$ such that $x_{n}^{\sim}\to y_0$. It follows that $x_0 x_n < t_0$ for all $n$ and thereby (since $\abs{x_n} \to \infty$), that $x_0 y_0 \leq 0$ for all $y_0 \in D$. On the other hand, $\esl{Mx_n + Q}\to My_{0}$ for the unbounded sequence $(Mx_n + Q)_{n \geq 1}$ (which is $\Pfs$ a subset of $\supp R$) implies $My_0 \in D$ $\Pfs$ and therefore
 \[ \P{x_0 M y_0 \leq 0}=\P{\esl{x_0 M} y_0 \leq 0}=1, \] 
in particular $\Prob(\esl{x_{0}M}\not\in B_{\delta}(y_{0}))=0$ for sufficiently small $\delta>0$ which is a contradiction to \eqref{A4} (with $n_0=1$).
\end{proof}

%\nocite{Goldie1991,Gui2006,Buracz2009,Klueppelberg2004,deSaporta2004} %Inhalt des Literaturverzeichnisses sicherstellen - alle hier aufgefuehrten tauchen auf


\begin{thebibliography}{10}
\providecommand{\url}[1]{\texttt{#1}}
\providecommand{\urlprefix}{URL }

\bibitem{Als1997}
G. Alsmeyer, \emph{{The Markov renewal theorem and related results}}, Markov
  Proc. Rel. Fields 3 (1997), pp. 103--127.

\bibitem{Athreya1978}
K. Athreya and P. Ney, \emph{{A new approach to the limit theory of recurrent
  {M}arkov chains}}, Transactions of the American Mathematical Society 245
  (1978), pp. 493--501.

\bibitem{BouLac1985}
P.~Bougerol and J.~Lacroix.
\newblock {\em Products of random matrices with applications to {S}chr\"odinger
  operators}, volume~8 of {\em Progress in Probability and Statistics}.
\newblock Birkh\"auser Boston Inc., Boston, MA, 1985.

\bibitem{Buracz2009}
D. Buraczewski, E. Damek, Y. Guivarc'h, A. Hulanicki, and R. Urban,
  \emph{Tail-homogeneity of stationary measures for some multidimensional
  stochastic recursions}, Probab. Theory Related Fields 145 (2009), pp.
  385--420.

\bibitem{deSaporta2004}
B. {de Saporta}, Y. Guivarc'h, and {\'E}. Le~Page, \emph{On the
  multidimensional stochastic equation {$Y_{n+1}=A_n Y_n + B_n$}}, Comptes
  Rendus Mathematique 339 (2004), pp. 499--502.

\bibitem{DiacFreed1999}
P. Diaconis and D. Freedman, \emph{Iterated random functions}, SIAM Review 41
  (1999), pp. 45--76.

\bibitem{Dunford1958}
N. Dunford and J.T. Schwartz, \emph{Linear Operators, Part I, General Theory},Wiley 1958.

\bibitem{Elton1990}
J.H. Elton, \emph{A multiplicative ergodic theorem for {L}ipschitz maps},
  Stoch. Proc. Appl. 34 (1990), pp. 39--47.

\bibitem{Goldie1991}
C.M. Goldie, \emph{Implicit renewal theory and tails of solutions of random
  equations}, Ann. Appl. Probab. 1 (1991), pp. 126--166.

\bibitem{Gui2006}
Y. Guivarc'h, \emph{Heavy tail properties of stationary solutions of
  multidimensional stochastic recursions}, in \emph{Dynamics \& stochastics},
  IMS Lecture Notes Monogr. Ser., Vol.~48, Inst. Math. Statist., Beachwood, OH,
   2006, pp. 85--99.

\bibitem{Karlin1959}
S. Karlin, \emph{Positive operators}, J. Math. Mech. 8 (1959), pp. 907--937.

\bibitem{Kesten1973}
H. Kesten, \emph{Random difference equations and renewal theory for products of
  random matrices}, Acta Math. 131 (1973), pp. 207--248.

\bibitem{Kesten1974}
H. Kesten, \emph{{Renewal Theory for Functionals of a Markov Chain with general
  state space}}, The Annals of Probability 2 (1974), pp. 355--386.

\bibitem{Klueppelberg2004}
C. Kl{\"u}ppelberg and S. Pergamenchtchikov, \emph{{The tail of the stationary
  distribution of a random coefficient AR(q) model}}, Annals of Applied
  Probability 14 (2004), pp. 971--1005.

\bibitem{LePage1983}
{\'E}. Le~Page, \emph{Th\'eor\`emes de renouvellement pour les produits de
  matrices al\'eatoires. \'{E}quations aux diff\'erences al\'eatoires}, in
  \emph{S\'eminaires de probabilit\'es {R}ennes 1983}, Univ. Rennes I, Rennes,
  1983, p. 116.

\bibitem{Num1978}
E. Nummelin, \emph{A splitting technique for {H}arris recurrent {M}arkov
  chains}, Z. Wahrsch. Verw. Gebiete 43 (1978), pp. 309--318.

\end{thebibliography}
\end{document}